\documentclass[a4paper,reqno, 11pt]{amsart}
\usepackage{graphicx}
\usepackage{amsmath,amssymb,amsthm,amsfonts}
\usepackage{amssymb}
\usepackage[active]{srcltx}
\usepackage{color}
\newtheorem{thm}{Theorem}[section]

\newtheorem{lem}[thm]{Lemma}
\newtheorem{prop}[thm]{Proposition}
\newtheorem{rem}[thm]{Remark}
\newtheorem{defn}[thm]{Definition}

\newcommand{\R}{\mathbb R}

\newcommand{\N}{\mathbb N}

\newcommand{\n}{\nabla}

\newcommand{\D}{\hbox{D}}

\begin{document}

\parindent 0pc
\parskip 6pt
\overfullrule=0pt

\title[Second order regularity]{Second order regularity for solutions to anisotropic degenerate elliptic equations}

\author{Daniel Baratta$^{*}$, 
Luigi Muglia$^{*}$ and Domenico Vuono$^{*}$}
\address{$^{*}$Dipartimento di Matematica e Informatica, UNICAL, Ponte Pietro  Bucci 31B, 87036 Arcavacata di Rende, Cosenza, Italy}
\email{daniel.baratta@unical.it}
\email{muglia@mat.unical.it}
\email{domenico.vuono@unical.it}


\date{\today}


\maketitle

\date{\today}

\begin{abstract}
We consider solutions to degenerate anisotropic elliptic equations in order to study their regularity. In particular we establish second-order estimates and enclose regularity results for the stress field. All our results are new even in the euclidean case.
\end{abstract}

\section{Introduction}
The goal of this paper is to investigate the second order regularity properties of solutions of degenerate elliptic problems in a possibly anisotropic medium. 

Let $\Omega\subset \mathbb{R}^n$ be a domain and, for some $\alpha\in (0,1)$, let $A\in C_{loc}^{1,\alpha}(0,+\infty)$ and $H\in C_{loc}^{2,\alpha}(\mathbb{R}^n\setminus\{0\})$ be representing a Finsler norm (see Section \ref{preliminari} for details). Let us consider a source term $f$ satisfying
\begin{itemize}
\item[$(H_f)$] the source term $f\in W_{loc}^{1,\frac{N}{N-\gamma-s}}(\Omega)\cap C_{loc}^{0,\alpha}(\Omega)$ where $s\in (0,N-\gamma)$ and $\gamma<N-2$ if $N>2$ and $\gamma=0$ if $N=2$. If $\gamma=0$,  we consider $f\in  W^{1,1}_{loc}(\Omega)\cap C_{loc}^{0,\alpha}(\Omega)$.  
\end{itemize}
We are interested in the second order regularity of solutions to
\begin{equation} \label{eq:forte}
    -\operatorname{div}(A(H(\nabla u)) H(\nabla u)\nabla H(\nabla u)) = f(x), \qquad \text{in } \Omega.
\end{equation} 
In the isotropic case, i.e., when $H$ is the classical Euclidean norm $H(\xi)=|\xi|$, our equation reduces to the well known degenerate equation based on the Uhlenbeck structure,
\begin{equation}\label{isotro}
 -\operatorname{div} (A(|\nabla u|)\nabla u)=f(x), \qquad \text{in } \Omega,
\end{equation}
where
\begin{equation*}
 -1 < \inf_{t>0} \frac{tA'(t)}{A(t)} := m_{A} \le M_{A} := \sup_{t>0} \frac{tA'(t)}{A(t)} < + \infty.
\end{equation*}
By choosing $A(t)=t^{p-2}$ ($p>1$), \eqref{isotro} encompasses the inhomogeneous $p$-Laplace equation
\begin{equation}\label{lapl}
 -\Delta_pu:=-\operatorname{div} (|\nabla u|^{p-2}\nabla u)=f(x), \qquad \text{in } \Omega.
\end{equation}
As a result, we immediately recognize that problem \eqref{eq:forte} reduces to the Finsler $p$-Laplacian problem
\begin{equation}\label{laplani}
 -\Delta_p^Hu:=-\operatorname{div} (H(\nabla u)^{p-1}\nabla H(\nabla u))=f(x), \qquad \text{in } \Omega,
\end{equation}
whenever $A(t)=t^{p-2}$ ($p>1$).\\
Recalling that for the $p$-Laplace equation, solutions are understood in a weak meaning, it is unsurprising that a similar meaning must be applied to solutions of \eqref{eq:forte}. To properly define weak solutions to \eqref{eq:forte}, we refer the reader to Section \ref{preliminari} .

The Calderón-Zygmund theory in the quasilinear context is a complex and challenging subject, as discussed in \cite{1,lars1,6,7,9,lars2,12,13,17, 16,23,24}. Recent significant results are presented in \cite{DM,DM2,DM3}. 

To the best of our knowledge, second-order regularity for equations with Uhlenbeck structure is not completely understood including the case of problem \eqref{eq:forte}.\\
By standard regularity results (see \cite{CFV17, DiB, GT, Ul, Li1, Li2, To}) weak solutions to equation \eqref{eq:forte} belong to $C_{loc}^{1,\alpha}(\Omega)\cap C^2(\Omega\setminus Z_u)$ where $Z_u$ denotes the set where the gradient vanishes.

In particular,  $W^{1,p}(\Omega)\cap L^{\infty}(\Omega)$ can be chosen as a natural space for the existence of solutions under a suitable set of assumptions.  Within this setting  Cozzi et al. in \cite{CFV17} proved that $u\in C^{1,\alpha}(\Omega)\cap C^3(\{|\nabla u|\neq 0\})$. 

Subsequently,  Castorina, Riey and Sciunzi \cite{CaRiSc} proved that if $u\in C^1(\overline{\Omega})$ is a weak solution of \eqref{eq:forte} then
\begin{itemize}
 \item if $p\in(1,3)$, $u\in W^{2,2}(\Omega)$;
 \item if $p\geq 3$ and $f$ is positive then $u\in W^{2,q}(\Omega)$ where $q<\frac{p-1}{p-2}$.
\end{itemize}
This holds under assumptions similar to those in \cite{CFV17} but defining and working on the \emph{linearized problem}.
 
The hypothesis  $u\in C^1(\overline{\Omega})$ is justified by assuming that $\Omega$ is smooth, which allows for the application of Lieberman's results in \cite{Li1}.

In \cite{ACF}, the authors proved that the stress field $H(\nabla u)^{p-1}\nabla H(\nabla u)\in W^{1,2}_{loc}(\Omega)$ if $u\in W^{1,p}_{loc}(\Omega)$ is a weak solution of \eqref{laplani} and $f\in L^q_{loc}(\Omega)$ for a suitable choice of $q$.  \\
For $p\in (1,2)$  and $f\in L^{r}_{loc}(\Omega)$ ($r>n$)\footnote{This result is a special case of a more general one  involving a source term 
$f$ that satisfies weaker integrability conditions.},  they obtain additional regularity, namely $u\in W^{2,2}_{loc}(\Omega)\cap C_{loc}^{1,\beta}(\Omega)$.  Furthermore,  for any $p>1$, $u\in W^{2,2}_{loc}(\{\nabla u|\neq 0\})\cap C_{loc}^{1,\beta}(\Omega)$.

Global second-order estimates for anisotropic problems, within the Uhlenbeck structure have been obtained in very recent result due to Antonini, Cianchi, Ciraolo, Farina and Maz'ya \cite{ACCFM}. This paper focuses on an $L^2$-second-order regularity theory for solutions to \eqref{eq:forte} which can be stated as
$$
 f\in L^2 \iff A(H(\nabla u)) H(\nabla u)\nabla H(\nabla u)\in W^{1,2}(\Omega).
$$
under minimal assumptions on the regularity of $\partial \Omega$ and on $H$.

We also mention that fine results regarding the Sobolev regularity of the stress field can be found in \cite{BaCiDiMa,beni,Cma,MMS,MMSV,S1,S2}.

In this paper, taking into account all the aforementioned results, we adopt an Orlicz-Sobolev space  defined by the function $A(\cdot)$ as a natural setting for solutions to \eqref{eq:forte}. This approach has been extensively used in \cite{CiMa11, CiMa14, CiMa18}.  We then prove that if additional regularity is imposed on $f$, further regularity results can be obtained for the solution.

\begin{thm}\label{W2}
    Let $u$ be a weak solution to \eqref{eq:forte}, where $f$ satisfies $(H_f)$. 
     Suppose $\displaystyle \inf_{t\in [0,M]}A(t)=0$. Then
    \begin{itemize}
        \item if $M_A < 1,$ it holds
        \begin{equation}\label{eq:1}
            u \in W^{2,2}_{loc}(\Omega).
        \end{equation}
        \item If $M_A \ge 1$   and if we assume that 
    \begin{equation}
        f \ge \tau > 0 \quad  \text{a.e. in    } \Omega.
    \end{equation} it holds
        \begin{equation}\label{eq:2}
            u \in W^{2,q}_{loc}(\Omega),
        \end{equation}
        with $\displaystyle q < \frac{M_A+1}{M_A}.$
    \end{itemize}
If $\displaystyle \inf_{t\in [0,M]}A(t)>0$, then \eqref{eq:1} holds without further assumptions on $M_A$.
\end{thm}
\begin{rem}
Notice that our results broaden those obtained in \cite{ CMMS, EST} in the isotropic case.  Indeed, one can easily observe that, compared to \cite{ CMMS, EST}, we do not impose any additional polynomial constraint on the function $A$ (see \cite {CMMS}-(1.3-1.4)-and \cite{EST}-(1.4-1.5)).  Our results depend only on the properties of $A$ itself.
\end{rem}
The reader may note that, as highlighted in \cite[Remark 2.7]{CiMa18} (see \cite[Remark 2.6]{CiMa18} for sake of completeness) if $f\in L^{n,1}$ and $\displaystyle \inf_t A(t)>0$ then $u\in W_{loc}^{2,2}(\Omega)$; we cover this result, for our source term (see Remark \ref{cianchirem}). \\
Here we investigate the complementary case $\displaystyle \inf_t A(t)=0$, mainly linking the regularity of the solution to the value $M_A$. As a rule, for equation \eqref{laplani}, $M_A=m_A=p-2$, we deduce that $u\in W^{2,2}_{loc}$ for $p\in[2,3)$ under our hypothesis on $f$ (recall \cite{ACF}). 

Our second result concerns the regularity of the families of stress fields $A(H(\nabla u))^{\beta} \, H(\nabla u)\nabla H(\nabla u)$-like. This result, once again, depends only on $m_A$ and $M_A$.

 \begin{thm}\label{eq:campovettoriale}
    Let $u$ be a weak solution to \eqref{eq:forte}, where $f$ satisfies $(H_f)$. \\Then 
    \begin{itemize}
    \item If $0 < m_{A} < M_{A},$ then 
    \begin{equation*} 
        A(H(\nabla u))^{\frac{k-1}{m_{A}}} \, H(\nabla u)\nabla H(\nabla u)\in W^{1,2}_{loc}(\Omega,\R^{N}),
    \end{equation*}
    for any $\displaystyle k > 1 + \frac{m_{A}}{2} - \frac{m_{A}}{2M_{A}}.$ \newline
    \item If $m_{A} < M_{A} < 0,$ then 
    \begin{equation*}
        A(H(\nabla u))^{\frac{k-1}{m_{A}}} \, H(\nabla u)\nabla H(\nabla u)\in W^{1,2}_{loc}(\Omega,\R^{N}),
    \end{equation*}
    for any $\displaystyle k > \frac 12+\frac{m_{A}}{2}.$ \newline
    \item If $m_{A} < 0 < M_{A},$ then 
    \begin{equation*}
        A(H(\nabla u))^{\frac{k-1}{M_{A}}} \, H(\nabla u)\nabla H(\nabla u)\in W^{1,2}_{loc}(\Omega,\R^{N}),
    \end{equation*}
    for any $\displaystyle k \in \left(\frac 12+\frac{M_A}{2}, 1 + \frac {M_A}{2}-\frac{M_A}{2m_A}\right).$ 
    \newline
    \item If $m_{A} = 0$ and $M_{A} > 0,$ then
    \begin{equation*}
        A(H(\nabla u))^{k-1} \, H(\nabla u)\nabla H(\nabla u) \in W^{1,2}_{loc}(\Omega,\R^{N}),
    \end{equation*}
    for any $\displaystyle k > \frac{3}{2}-\frac{1}{2M_A}.$\newline
    \item If $m_{A} < 0$ and $M_{A} = 0$, then 
    \begin{equation*}
        A(H(\nabla u))^{k-1} \, H(\nabla u)\nabla H(\nabla u) \in W^{1,2}_{loc}(\Omega,\R^{N}),
    \end{equation*} for any $\displaystyle k<\frac{3}{2}-\frac{1}{2m_A}.$
    \end{itemize}
\end{thm}
Of course, appropriate choices for $k$ recover the optimal results due to Antonini et al. \cite{ACCFM}
$$
A(H(\nabla u)) H(\nabla u)\nabla H(\nabla u)\in W^{1,2}_{loc}(\Omega).
$$

Concluding the introduction, let us highlight that the study of anisotropic operators is of significant interest for two main reasons. The first concerns the pure mathematics and the richer geometric structure that the anisotropy induces. The second relies applications, from the computer vision to the continuum mechanics, particularly in scenarios where materials exhibit distinct behaviors depending on directionality, often due to the crystalline microstructure of the medium.
We refer to \cite{BrRiSo,CaRiSc,CFV16,CFV17}  for complete details and references.

The paper is organized as follows. In the next section we introduce the notation, preliminaries and a brief introduction to Orlicz-Sobolev spaces, to clarify the meaning of weak solution for \eqref{eq:forte}. In Section \ref{sezz3} we establish summability properties for the second derivative which are fundamental for our results. In the same Section we prove Theorem \ref{eq:campovettoriale} by means of Propositions. In the Section \ref{sezz4} we get summability properties for the weight $A(H(\nabla u))^{-1}$ with respect to a kernel $|x-y|^{-\gamma}$. This is a powerful tool for proving the second statement of Theorem \ref{W2} and plays a key role in establishing weighted Sobolev inequalities (see \cite{DaSc}). The proof of Theorem \ref{W2} concludes the paper.

\section{Notation and preliminaries}\label{preliminari}
Generic numerical constants will
be denoted by $C$ (with subscript in some case) and they will be
allowed to vary within a single line or formula. The Lebesgue measure of a measurable set $K$
will be denoted by $|K|$ .\\ 

In what follows, we will assume standard hypotheses on the functions involved in our problem (see 	\cite{CiMa11, CiMa14, CiMa18}).\\
The function $A:(0,+ \infty) \rightarrow (0,+\infty)$ is of class $C_{loc}^{1,\alpha}((0,+\infty))$ and fulfills the following:
 \begin{equation} \label{eq:hp principale}
  -1 < \inf_{t>0} \frac{tA'(t)}{A(t)} := m_{A} \le M_{A} := \sup_{t>0} \frac{tA'(t)}{A(t)} < + \infty.
 \end{equation}
We recall that above assumption gives us that (see \cite[Proposition $4.1$]{CiMa14})
\begin{equation}\label{eq:hp cianchi}
 A(1) \min \{ t^{m_A}, t^{M_A} \} \le A(t) \le A(1) \max \{t^{m_A}, t^{M_A}   \} \qquad \text{for } t>0,
\end{equation}
holds.

 Moreover, since $m_A>-1$ there exists $\sigma\in[0,1)$ such that $m_A+\sigma>0$ and 
 \begin{equation}\label{sigma}
     t^\sigma A(t)\to 0, \quad\text{as } t\to 0.
 \end{equation} 

The function $H$ will be used to indicate the well-known Finsler norm; by definition, $H$ satisfies
\begin{enumerate}
 \item[(i)] $H$ is a norm of class $C_{loc}^{2,\alpha}(\R^N\setminus \{0\})$;
 \item[(ii)] $H$ is uniformly elliptic, that is, the set $\mathcal{B}_1^H:=\{\xi \in \R^N  :  H(\xi) < 1\}$ is \emph{uniformly convex}, i.e., all the principal curvatures of its boundary are bounded away from zero. This is equivalent to say that there exists $\lambda>0$ such that, for all $\xi\neq 0$ and $\eta\in\mathbb{R}^N$,
\begin{equation}\label{H2 da sotto}
 \langle D^2 H^2(\xi)\eta,\eta\rangle \geq \lambda |\eta|^2
\end{equation}
\end{enumerate}
\begin{rem} Next we summarize some useful properties of the function $H$ proved in \cite{CFV16,CFV17}.
\begin{itemize}
\item $H$ is a norm equivalent to the Euclidean one, i.e. there exist $c_1,c_2>0$ such that:
\begin{equation}\label{H equiv euclidea}
c_1|\xi|\leq H(\xi)\leq c_2|\xi|,\quad\forall\, \xi\in\R^N.
\end{equation}

\item $H$ is 1-homogeneous, hence, by the Euler Theorem it follows
\begin{equation}\label{eq:eulero}
 \langle\nabla H(\xi), \xi\rangle = H(\xi), \qquad\forall\, \xi \in \R^N\setminus \{0\}.
\end{equation}
Moreover
\begin{equation}\label{eq:grad 0 omog}
\nabla H(t\xi)=\hbox{sign}(t)\n H(\xi), \qquad\forall\,\xi \in \R^N\setminus \{0\},\, \forall t\neq 0.
\end{equation}
By the previous equality, we infer that there exists $K_1>0$ such that 
\begin{equation}\label{grad H limitato}
|\nabla H(\xi)| \le  K_1, \qquad \forall \xi\in\R^N\setminus \{0\}.
\end{equation}
\item Since $H^2$ is 2-homogeneous there exists $\Lambda>0$ such that
$$
 \langle D^2 H^2(\xi)\eta,\eta\rangle \leq \Lambda_2 |\eta|^2
$$
for $\xi\neq 0$ and $\eta\in\mathbb{R}^N$.
\item Since $D^2H$ is a $(-1)-$homogeneous function we have
\begin{equation}\label{eq:hess -1 omog}
\D^2 H(t\xi)=\frac{1}{|t|}\D^2H(\xi) \qquad\forall\,\xi \in  \R^N\setminus \{0\},\, \forall t\neq 0.
\end{equation}
It immediately implies that there exists $K_2 >0$ such that 
\begin{equation} \label{eq:Hess H limitato}
    |D^{2}H(\xi)| \le \frac{K_2}{|\xi|}, \qquad \forall \xi\in\R^N\setminus \{0\},
\end{equation}
where $|\cdot|$ denotes the usual Euclidean  norm of a matrix, and 
\begin{equation}\label{eq:hess nullo}
\D^2 H(\xi)\xi=0, \qquad\forall\,\xi\in\R^N\setminus \{0\}.
\end{equation}

\item Assumption (ii) is equivalent to state:
\begin{equation} \label{eq:definita positiva 2}
    \exists \Lambda > 0: \langle D^2 H(\xi)v, v \rangle \geq \Lambda |\xi|^{-1} |v|^2 \quad \forall \xi \in \R^{N} \setminus \{0\}, \; \forall v \in \nabla H(\xi)^\bot.
\end{equation}
\end{itemize}
\end{rem}

Next Lemma will be useful later.
\begin{lem} \label{eq:lemma} For any $v,w \in \R^{N}$ and for any $\xi \in \R^{N}\setminus \{0\}$, there exist two constants $\tilde{C}, \Bar{C} > 0$ such that:
 \begin{equation}\label{eq:minoraz}
   \left[A(H(\xi)) + H(\xi)A'(H(\xi)) \right] \langle \nabla H(\xi),v \rangle^{2} +
   H(\xi) A(H(\xi)) \langle D^{2}H(\xi) v, v \rangle \ge \Tilde{C} A(H(\xi)) |v|^{2} 
 \end{equation}\\
\begin{equation}\label{eq:maggioraz}
\begin{split}
    \left[  A(H(\xi)) + H(\xi)A'(H(\xi)) \right] \langle \nabla H(\xi), v \rangle\langle \nabla H(\xi), w \rangle &+ H(\xi) A(H(\xi)) \langle D^{2}H(\xi) v, w \rangle \\
    &\le \Bar{C} A(H(\xi)) |v| |w| \\
\end{split}
\end{equation}

where $\Tilde{C}=\Tilde{C}(c_1, m_{A},\Lambda)$ and $\Bar{C}=\Bar{C}(c_2,K_1,K_2,M_{A}).$ 
\end{lem}

\begin{proof}
    Let $\xi \in \R^{N}\setminus \{0\}$. By  \eqref{eq:eulero}, we deduce that $\R^{N} = Span\{\xi, \nabla H(\xi)^{\bot}\}$. Now we consider
    \begin{equation*}
        v = \alpha \xi + \eta,
    \end{equation*}
    with $\alpha \in \R$ and $\eta \in \nabla H(\xi)^{\bot}.$
    By \eqref{eq:hp principale}, \eqref{eq:eulero}, \eqref{eq:hess nullo}, \eqref{eq:definita positiva 2}, since $D^2H$ is a symmetric matrix we get:
    \begin{eqnarray}\label{eq:blabla1}
        \nonumber &&\left[  A(H(\xi)) + H(\xi)A'(H(\xi)) \right] \langle \nabla H(\xi),v \rangle^{2} + H(\xi) A(H(\xi)) \langle D^{2}H(\xi) v, v \rangle \\
        \nonumber &=&\left[  A(H(\xi)) + H(\xi)A'(H(\xi)) \right] \left(\alpha\langle  \nabla H(\xi), \xi \rangle \right)^{2} + \alpha H(\xi) A(H(\xi))  \langle D^{2}H(\xi) v,  \xi  \rangle \\
        &&+H(\xi) A(H(\xi))  \langle D^{2}H(\xi) v,\eta\rangle \\
        \nonumber &=&\left[  A(H(\xi)) + H(\xi)A'(H(\xi)) \right] \alpha^2 H^2(\xi)+ \alpha H(\xi) A(H(\xi)) \langle D^{2}H(\xi) \eta, \xi \rangle \\
        \nonumber &&+H(\xi) A(H(\xi)) \langle D^{2}H(\xi) \eta, \eta  \rangle \\ \nonumber &=&\left[  A(H(\xi)) + H(\xi)A'(H(\xi)) \right] \alpha^2 H^2(\xi)+H(\xi) A(H(\xi)) \langle D^{2}H(\xi) \eta, \eta  \rangle \\
        \nonumber &\geq&\left[  A(H(\xi)) + H(\xi)A'(H(\xi)) \right] \alpha^2 H^{2}(\xi) + \Lambda H(\xi) A(H(\xi)) |\xi|^{-1} |\eta|^{2} =\\
        \nonumber &=&\alpha^2 A(H(\xi)) \left[1 + \frac{H(\xi)A'(H(\xi))}{A(H(\xi))}\right] H^{2}(\xi) + \Lambda H(\xi) A(H(\xi)) |\xi|^{-1} |\eta|^{2} \\
        \nonumber &\geq&\alpha^2 \left( m_{A}+1 \right) A(H(\xi)) H^{2}(\xi)  + \Lambda H(\xi) A(H(\xi)) |\xi|^{-1} |\eta|^{2}
     \end{eqnarray} 
    We now consider two cases: if $2|\alpha\xi| \le |v|$, 
       \begin{equation} \label{eq:primocaso}
        |
\eta|^{2} = |v - \alpha \xi|^{2} \ge (|v| - |\alpha \xi|)^{2} \ge \frac{|v|^{2}}{4}.
    \end{equation}
Therefore, by \eqref{eq:primocaso} and \eqref{H equiv euclidea} we have that 
 \begin{eqnarray}
  \nonumber &&\left[  A(H(\xi)) + H(\xi)A'(H(\xi)) \right] \langle \nabla H(\xi),v \rangle^{2} + H(\xi) A(H(\xi)) \langle D^{2}H(\xi) v, v \rangle \\
  \nonumber &\ge& \Lambda H(\xi) A(H(\xi))  |\xi|^{-1} |\eta|^{2}  \ge  \Lambda H(\xi) A(H(\xi)) |\xi|^{-1} \frac{|v|^{2}}{4} \ge c_{1} \Lambda A(H(\xi))  \frac{|v|^{2}}{4} \\
  &=& \Tilde{C_{1}} A(H(\xi)) |v|^{2},
\end{eqnarray}
where $\displaystyle \Tilde{C_{1}} := \frac{c_{1} \Lambda}{4} > 0.$\\ \\
 On the contrary, if $2|\alpha\xi| > |v|$:
 \begin{eqnarray}
\nonumber &&\left[  A(H(\xi)) + H(\xi)A'(H(\xi)) \right] \langle \nabla H(\xi),v \rangle^{2} + H(\xi) A(H(\xi)) \langle D^{2}H(\xi) v, v \rangle \\
 &\ge& \alpha^2c_1 (m_{A}+1) A(H(\xi)) |\xi|^{2} \ge \Tilde{C_{2}} A(H(\xi)) |v|^{2},
\end{eqnarray}
 where $\displaystyle \Tilde{C_{2}} := \frac{c_1(m_{A}+1)}{4} > 0$.\\ \\
 Choosing $\Tilde{C} = \min \{\Tilde{C_1} , \Tilde{C_2}\},$ we have \eqref{eq:minoraz}.\\ 
 
 Let us now prove \eqref{eq:maggioraz}. By \eqref{eq:hp principale}, \eqref{H equiv euclidea}, \eqref{grad H limitato} and \eqref{eq:Hess H limitato}  we get
    \begin{equation}
        \begin{split}
            & \quad \left[  A(H(\xi)) + H(\xi)A'(H(\xi)) \right] \langle \nabla H(\xi),v\rangle \langle\nabla H(\xi), w \rangle + H(\xi) A(H(\xi)) \langle D^{2}H(\xi) v, w \rangle  \\
            &\le \left[  A(H(\xi)) + H(\xi)A'(H(\xi)) \right]  |\nabla H(\xi)|^2  |v|  |w|  + H(\xi) A(H(\xi))  |D^{2}H(\xi)| |v| |w|   \\
            &\le K^{2}_{1} \left(1 +|M_{A}| \right) A(H(\xi))  |v||w| + K_{2}c_2 A(H(\xi)) |v||w| \\
            &=: \Bar{C}(c_2,K_1,K_2, M_A) A(H(\xi)) |v||w|, \\
        \end{split}
    \end{equation}
   where $\Bar{C}(c_2,K_1,K_2, M_{A})>0$. This completes the proof.
\end{proof}

\subsection{Weak solution of our problem.}
We conclude this section by recalling the concept of weak solution for our main problem \eqref{eq:forte}.
Following \cite[Section 2.2]{CiMa11} and assuming \eqref{eq:hp principale} for $A$, it can be verified that the function
\begin{equation}\label{funzioneA}
  \mathcal{A}(t):=\int_0^t A(s)s \, ds  
\end{equation}
is a Young function (i.e. it is convex and $\mathcal A(0)=0$) and
$$
 \inf_{t>0}\left(1+\frac{tA'(t)}{A(t)}\right)>0,
$$
that is equivalent to the so-called $\nabla_2$-condition in the theory of Young functions (see \cite{ACCFM}).
 By \cite[Proposition 2.9]{CiMa11}, under the assumption \eqref{eq:hp principale} for $A$,  
 \begin{itemize}
 \item $\mathcal A\in\Delta_2$ (i.e. $\mathcal A(2t)\leq C \mathcal A(t)$ whenever $t>0$).
 \item $\tilde{\mathcal A}\in\Delta_2$, where  $\displaystyle \tilde{\mathcal A}(t):=\sup_{s\geq0}(st-\mathcal A(s))$.
 \item $\mathcal A$ is a $N-$function (i.e. $\mathcal A(t)/t$ is a null function as $t\to 0$ and it goes to $+\infty$, as $t\to +\infty$).
 \end{itemize}
For a such function $\mathcal A$ and $\Omega$ a domain in $\mathbb{R}^N$, following, \cite{CiMa11}, let us define
\begin{equation}
W^{1,\mathcal A}(\Omega)=\{u\in L^{\mathcal A}(\Omega): u\mbox{ is weakly-differentiable and }|\nabla u|\in L^{\mathcal A}(\Omega)\}
\end{equation}
where $L^{\mathcal A}(\Omega)$ is the Banach space of the real-valued measurable functions on $\Omega$ such that the Luxemburg norm
$$
\|u\|_{\mathcal A}=\inf \left\{\mu>0: \int_{\Omega} \mathcal A\left(\frac{|u(x)|}{\mu}\right) \ dx\leq 1\right \}<\infty.
$$

The space $W_0^{1,\mathcal A}(\Omega)$ is defined as the closure of $C_c^\infty(\Omega)$ in $W^{1,\mathcal A}(\Omega)$; the space  $W_{loc}^{1,\mathcal A}(\Omega)$ is defined accordingly.

The next proposition summarizes some density results about $W^{1,\mathcal{A}}(\Omega)$  (see \cite[Theorem 8.28]{Ad} and \cite[Theorem 2.1]{CiMa11}).
\begin{prop}\label{corDT} 
Let  $\mathcal{A}$ be defined in \eqref{funzioneA} under our assumptions on $A$. Then 
\begin{itemize}
\item $W^{1,\mathcal A}(\Omega)$ is reflexive.
  \item $C^\infty(\Omega)\cap W^{1,\mathcal A}(\Omega)$ is dense in $W^{1,\mathcal A}(\Omega)$ . 
  \item The  space $C_c^\infty(\mathbb{R}^N)$ is dense in $W^{1,\mathcal A}(\mathbb{R}^N)$ . 
  \item If $\Omega'$ is bounded and it has a Lipschitz boundary, $C^\infty(\overline{\Omega'})$ is dense in $W^{1,\mathcal A}(\Omega')$.
\end{itemize}
\end{prop}

\begin{defn} Let $\Omega\subset \mathbb{R}^n$ be a domain and $f$  satisfying $(H_f)$. A weak solution of \eqref{eq:forte} is a function $u\in W_{loc}^{1,\mathcal A}(\Omega)$ such that it holds
\begin{equation}\label{eq:debole}
 \int_{\Omega}A(H(\nabla u))  H(\nabla u) \langle \nabla H(\nabla u) , \nabla \psi \rangle \,dx = \int_{\Omega} f \psi \,dx,
\end{equation}
 for every $\psi \in C^{\infty}_{c}(\Omega).$

 By density arguments,  we can test our problem against $\psi\in W^{1,\mathcal A}_0(\Omega')$, for any subset $\Omega'\subset\subset \Omega$.
\end{defn}

\section{Local regularity}\label{sezz3}
In this section, we establish a result regarding the integrability of the second derivative of our solution and subsequently use it to prove Theorem \ref{eq:campovettoriale}.

We start by outlining a regularity property of the weak solutions to \eqref{eq:forte}.

\begin{prop}
Let $u\in W_{loc}^{1,\mathcal A}(\Omega)$ be a weak solution of \eqref{eq:forte}. Then we have that $u\in C^{1,\alpha}_{loc}(\Omega)\cap C^2(\{\nabla u\neq 0\})$, for some $\alpha \in (0,1)$.  
\end{prop}
\begin{proof}
     By Lemma \ref{eq:lemma} (see \eqref{eq:minoraz}) and \cite[see $(3.29)$]{ACCFM} we obtain
\begin{eqnarray*}
   &&\left[A(H(\xi)) + H(\xi)A'(H(\xi)) \right] \langle \nabla H(\xi),v \rangle^{2} +
   H(\xi) A(H(\xi)) \langle D^{2}H(\xi) v, v \rangle \\
   &\ge& \Tilde{C} A(H(\xi)) |v|^{2}\geq \Tilde{C}\frac{c(m_a,\lambda,\Lambda_2)}{c_2} A(|\xi|) |v|^{2},
 \end{eqnarray*}
therefore hypotheses (1.10) of Lieberman \cite{Li2} are satisfied. Moreover, by \cite[Theorem $5.1$]{Kor} we have that $u\in L^\infty_{loc}(\Omega)$. These permits to exploit \cite[Theorem 1.7]{Li2} to get $u\in C^{1,\alpha}_{loc}(\Omega) $.  \\
Now let us consider a compact set $K\subset\subset (\Omega\setminus Z_u)$, where $Z_u$ denotes the set where the gradient vanishes. By \cite[Proposition 1]{To} we have that $u\in W^{2,2}_{loc}(K)$ (see also \cite[Theorem 2.1]{ACCFM}). Then we may apply
\cite[Theorem 6.3]{Ul}-page 283 to obtain that $u\in C^2(\Omega\setminus Z_u)$.  
\end{proof}

Since we have that $u \in C^{1,\alpha}_{loc}(\Omega) \cap C^{2}(\Omega\setminus Z_{u})$ let us denote by
\begin{equation}
\label{DerSeconde}
\tilde{u}_{ij}(x):=
\begin{cases}
\displaystyle {u}_{x_{i}x_{j}}(x) & x\in \Omega \setminus Z_{u},
\\
0 & x\in Z_{u}.
\end{cases}
\end{equation}
Our first result give us the summability of the second derivative in the above sense.  In order to make the proof readable,  we usually omit the tilde-symbol on the derivatives.

Moreover, in what follows we will denote by $u_i:=\partial_{x_i}u$.\\

The following integrability property is crucial in our method for investigating the regularity of weak solutions to \eqref{eq:forte}.
\begin{thm}  \label{stime derivate seconde}
    Let $u\in W_{loc}^{1,\mathcal A}(\Omega)$ be a weak solution of \eqref{eq:forte}, where $f$ satisfies $(H_f)$.\\
     Fix $x_0 \in \Omega$ and $R>0$ such that $B_{2R}(x_0) \subset \subset \Omega,$ and consider $y \in \Omega.$ Then, for $0 \le \beta < 1$ and $\gamma < N-2$ for $N \ge 3$ ($\gamma = 0$ for $N = 2$), we have:
    \begin{equation}\label{eq:stimadersec}
        \int_{B_{R}(x_0)\setminus Z_{u}} \frac{A(H(\nabla u)) |\nabla u_i|^2}{|x-y|^{\gamma}|u_i|^{\beta}} \ dx \le C \quad \forall \, i=1,...,N,
    \end{equation}
    where $C =C(c_1,c_2, K_1,K_2, m_{A}, M_{A}, x_0, R, \beta, \gamma, \Lambda, \|\nabla u\|_{L^{\infty}_{loc}(\Omega)},\|f\|_{1,\frac{N}{(N-\gamma)^-}})$ is a positive constant.
\end{thm}

\begin{proof}
    For $\varphi \in C^{\infty}_{c}(\Omega\setminus Z_{u})$, let us denote $\varphi_i:=\partial_{x_i}\varphi$. Using $\varphi_i$ as a test function in \eqref{eq:debole}, since $f(x)\in W^{1,\frac{N}{(N-\gamma)^-}}(\Omega)$, integrating by part we have
   \begin{eqnarray}\label{linearizzato}
      \nonumber   L_{u}(u_i,\varphi) &:=& \int_{\Omega} A(H(\nabla u)) \langle \nabla H(\nabla u) , \nabla u_i \rangle \langle \nabla H(\nabla u) , \nabla \varphi \rangle \ dx \\ &&+ \int_{\Omega} A'(H(\nabla u)) H(\nabla u) \langle \nabla H(\nabla u) , \nabla u_i \rangle \langle \nabla H(\nabla u) , \nabla \varphi \rangle \ dx \\ \nonumber &&+ \int_{\Omega} A(H(\nabla u)) H(\nabla u) \langle D^{2} H(\nabla u) \nabla u_i, \nabla \varphi \rangle \ dx  - \int_{\Omega} f_i \varphi \ dx = 0.
   \end{eqnarray}
    Let us now take an arbitrary $\epsilon > 0$ and define for $t\geq 0$:
    \begin{equation} \label{eq:Geps}
        G_{\epsilon}(t) := 
        \begin{cases}
            0               & \text{if} \quad t \in [0, \epsilon]\\
            (2t -2\epsilon) & \text{if} \quad t \in [\epsilon, 2\epsilon]\\
            t               & \text{if} \quad t \in [2\epsilon, \infty),\\
        \end{cases}
    \end{equation}
    while $G_{\epsilon}(t) := -G_{\epsilon}(-t)$ if $t \le 0.$ Moreover we consider a cut-off function $\varphi_{R} := \varphi \in C^{\infty}_{c}(B_{2R}(x_0))$ such that:
    \begin{equation} \label{eq:varphi}
        \begin{cases}
        \varphi = 1 & \text{in} \quad B_{R}(x_0)\\
        |\nabla \varphi| \le \frac{2}{R} & \text{in} \quad B_{2R}(x_0)\setminus B_{R}(x_0)\\
        \varphi = 0 & \text{in} \quad B_{2R}(x_0)^{c}.\\
    \end{cases}
    \end{equation}
    For $0 \le \beta < 1,$ $\gamma < N-2$ if $N \ge 3$ \footnote{$\gamma = 0$ if $ N = 2$.} and for every $\epsilon,\delta > 0$ we set:
    \begin{equation} \label{eq:Hdelta}
        T_{\epsilon}(t) := \frac{G_{\epsilon}(t)}{|t|^{\beta}}, \quad \quad H_{\delta}(t) := \frac{G_{\delta}(t)}{|t|^{\gamma +1}}.
    \end{equation}
    Next, we test \eqref{linearizzato} by the function
    \begin{equation}
        \psi(x) := T_{\epsilon}(u_i)H_{\delta}(|x-y|)\varphi^{2}(x)
    \end{equation}
 obtaining:
    \begin{equation}\label{l1}
    \begin{split}
        & \int_{\Omega} A(H(\nabla u)) \langle \nabla H(\nabla u) , \nabla u_i \rangle \langle \nabla H(\nabla u) , \nabla u_i \rangle T'_{\epsilon}(u_i) H_{\delta} \varphi^{2} \ dx \\ 
        &+ \int_{\Omega}  A(H(\nabla u)) \langle \nabla H(\nabla u) , \nabla u_i \rangle \langle \nabla H(\nabla u) , \nabla H_{\delta} \rangle T_{\epsilon}(u_i)  \varphi^{2} \ dx \\ 
        &+ 2 \int_{\Omega}  A(H(\nabla u)) \langle \nabla H(\nabla u) , \nabla u_i \rangle \langle \nabla H(\nabla u) , \nabla \varphi \rangle T_{\epsilon}(u_i) H_{\delta}  \varphi \ dx \\ 
        &+ \int_{\Omega} H(\nabla u) A'(H(\nabla u)) \langle \nabla H(\nabla u) , \nabla u_i \rangle \langle \nabla H(\nabla u) , \nabla u_i \rangle T'_{\epsilon}(u_i) H_{\delta} \varphi^{2} \ dx \\ 
        &+ \int_{\Omega} H(\nabla u) A'(H(\nabla u)) \langle \nabla H(\nabla u) , \nabla u_i \rangle \langle \nabla H(\nabla u) , \nabla H_{\delta} \rangle T_{\epsilon}(u_i) \varphi^{2} \ dx \\ 
        &+2 \int_{\Omega} H(\nabla u) A'(H(\nabla u)) \langle \nabla H(\nabla u) , \nabla u_i \rangle \langle \nabla H(\nabla u) , \nabla \varphi \rangle T_{\epsilon}(u_i) H_{\delta} \varphi \ dx \\
        &+ \int_{\Omega} H(\nabla u) A(H(\nabla u)) \langle D^{2}H(\nabla u) \nabla u_i, \nabla u_i \rangle T'_{\epsilon}(u_i) H_{\delta} \varphi^{2} \ dx \\ 
        &+ \int_{\Omega} H(\nabla u) A(H(\nabla u)) \langle D^{2}H(\nabla u) \nabla u_i, \nabla H_{\delta} \rangle T_{\epsilon}(u_i) \varphi^{2} \ dx \\ 
        &+2 \int_{\Omega} H(\nabla u) A(H(\nabla u)) \langle D^{2}H(\nabla u) \nabla u_i, \nabla \varphi \rangle T_{\epsilon}(u_i) H_{\delta}  \varphi \ dx = \int_{\Omega} f_i T_{\epsilon}(u_i) H_{\delta} \varphi^{2} \ dx .
    \end{split}
    \end{equation}
   Next we have to estimate every integral in the formula above, hence let us rename:
    \begin{equation}
    \begin{split}
        & I_{1} := \int_{\Omega} A(H(\nabla u)) |\langle \nabla H(\nabla u), \nabla u_i \rangle|^{2} T'_{\epsilon}(u_i) H_{\delta} \varphi^{2} \ dx \\ & 
        I_{2} := \int_{\Omega} H(\nabla u) A'(H(\nabla u)) |\langle \nabla H(\nabla u), \nabla u_i \rangle|^{2} T'_{\epsilon}(u_i) H_{\delta} \varphi^{2} \ dx \\ 
        &
        I_{3} := \int_{\Omega} H(\nabla u) A(H(\nabla u)) \langle D^{2} H(\nabla u) \nabla u_i, \nabla u_i \rangle T'_{\epsilon}(u_i) H_{\delta} \varphi^{2} \ dx \\ &
        I_{4} := \int_{\Omega}  A(H(\nabla u)) \langle \nabla H(\nabla u) , \nabla u_i \rangle \langle \nabla H(\nabla u) , \nabla H_{\delta} \rangle T_{\epsilon}(u_i) \varphi^{2} \ dx \\ &
        I_{5} :=2 \int_{\Omega}  A(H(\nabla u)) \langle \nabla H(\nabla u) , \nabla u_i \rangle \langle \nabla H(\nabla u) , \nabla \varphi \rangle T_{\epsilon}(u_i) H_{\delta} \varphi \ dx \\ &
        I_{6} := \int_{\Omega} H(\nabla u) A'(H(\nabla u)) \langle \nabla H(\nabla u) , \nabla u_i \rangle \langle \nabla H(\nabla u) , \nabla H_{\delta} \rangle T_{\epsilon}(u_i) \varphi^{2} \ dx \\ &
        I_{7} :=2 \int_{\Omega} H(\nabla u) A'(H(\nabla u)) \langle \nabla H(\nabla u) , \nabla u_i \rangle \langle \nabla H(\nabla u) , \nabla \varphi \rangle T_{\epsilon}(u_i) H_{\delta} \varphi \ dx \\ &
        I_{8} := \int_{\Omega} H(\nabla u) A(H(\nabla u)) \langle D^{2}H(\nabla u) \nabla u_i, \nabla H_{\delta} \rangle T_{\epsilon}(u_i) \varphi^{2} \ dx \\ &
        I_{9} :=2 \int_{\Omega} H(\nabla u) A(H(\nabla u)) \langle D^{2}H(\nabla u) \nabla u_i, \nabla \varphi \rangle T_{\epsilon}(u_i) H_{\delta} \varphi \ dx \\ &
        I_{10} := \int_{\Omega} f_i T_{\epsilon}(u_{i}) H_{\delta} \varphi^{2}\ dx,
    \end{split}
    \end{equation}
in such a way that \eqref{l1} is rewritten as
$$
 I_1+I_2+I_3=I_{10}-\sum_{k=4}^9 I_k.
$$
    Note that, by Lemma \ref{eq:lemma} (see in particular \eqref{eq:minoraz}), and since $T'_\varepsilon(t)\ge 0$, for $\beta<1$, we have:
    \begin{equation} \label{eq:sx}
        \Tilde{C}(c_1, m_{A},\Lambda) \int_{\Omega} A(H(\nabla u)) |\nabla u_i|^{2} T'_{\epsilon}(u_i) H_{\delta} \varphi^{2} \ dx \le I_{1} + I_{2} + I_{3}.
    \end{equation}

By \eqref{eq:maggioraz} in Lemma \eqref{eq:lemma}, by definition of $T_\varepsilon$ and using a weighted Young inequality we get:
    \begin{eqnarray}\label{eq:4-6-8}
       \nonumber  -(I_{4} + I_{6} + I_{8}) &=& -\int_{\Omega}  \, \big [(A(H(\nabla u))+H(\nabla u) A'(H(\nabla u))) \langle \nabla H(\nabla u),\nabla u_i \rangle \langle \nabla H(\nabla u),\nabla H_{\delta} \rangle \\
        \nonumber &&\qquad + H(\nabla u) A(H(\nabla u)) \langle D^{2}H(\nabla u) \nabla u_i, \nabla H_{\delta} \rangle \big ] T_{\epsilon}(u_i) \varphi^{2}  \ dx \\ 
\nonumber &\le& \Bar{C}(c_2, K_1, K_2, M_{A}) \int_{\Omega} A(H(\nabla u)) |\nabla u_i||\nabla H_{\delta}| |T_{\epsilon}(u_i)| \varphi^2 \ dx \\ 
\nonumber & \le& \tilde{C}\int_{\Omega} A(H(\nabla u)) |\nabla u_i| \frac{|T_{\epsilon}(u_i)|}{|x-y|^{\gamma + 1}}\varphi^{2}  \ dx \\ 
&\le&  \tilde{C}\int_{\Omega} \frac{A^{\frac{1}{2}}(H(\nabla u))|\nabla u_i|(|G_\varepsilon(u_i)|)^{\frac 12} \varphi}{|x-y|^{\frac{\gamma}{2}}|u_i|^{\frac{\beta+1}{2}}}\  \frac{A^{\frac{1}{2}}(H(\nabla u))|u_i|(|G_\varepsilon(u_i)|)^{\frac 12} \varphi}{|x-y|^{\frac{\gamma +2}{2}}|u_i|^{\frac{\beta+1}{2}}} \ dx \\ 
\nonumber &\le& \theta \int_{\Omega} \frac{A(H(\nabla u))|\nabla u_i|^{2}|G_\varepsilon(u_i)|}{|x-y|^{\gamma}|u_i|^{\beta}|u_i|} \varphi^{2} + C \int_{\Omega} \frac{A(H(\nabla u))|u_i|^{2-\beta}}{|x-y|^{\gamma + 2}} \varphi^{2}  \ dx \\ 
\nonumber &\le& \theta \int_{\Omega} \frac{A(H(\nabla u))|\nabla u_i|^{2}G_\varepsilon(u_i)}{|x-y|^{\gamma}|u_i|^{\beta}u_i} \varphi^{2} + C \int_{\Omega} \frac{A(H(\nabla u))|\nabla u|^{2-\beta}}{|x-y|^{\gamma + 2}} \varphi^{2} \ dx ,
    \end{eqnarray}
    where in the last inequality we use that $\displaystyle \frac{|G_\varepsilon(u_i)|}{|u_i|}=\frac{G_\varepsilon(u_i)}{u_i}$ and denoting by $C(c_2,K_1,K_2,M_A,\gamma,\theta)$ a positive constant. Now we set $B := B_{R}(x_0),$ and 
\begin{equation} \label{eq:M}
    M_\gamma := \max \biggl\{ \sup_{y \in \Omega} \int_{B} \frac{1}{|x-y|^{\gamma}} \ dx , \ \sup_{y \in \Omega} \int_{B} \frac{1}{|x-y|^{\gamma + 2}}  \ dx \biggr\}.
\end{equation}

We note that $M_\gamma$ does not depend on $y$.    
By \eqref{eq:hp cianchi}-\eqref{sigma}, we have that the function $t \mapsto tA(t)$ is locally bounded, so:
\begin{eqnarray}\label{zio}
    \nonumber \int_{\Omega} \frac{A(H(\nabla u))|\nabla u|^{2-\beta}}{|x-y|^{\gamma + 2}} \varphi^{2} \ dx  \le C(\|\nabla u\|_{L^{\infty}_{loc}(\Omega)}) \int_{B} \frac{1}{|x-y|^{\gamma + 2}} \ dx \le  C(m_A, M_A, M_\gamma,\|\nabla u\|_{L^{\infty}_{loc}(\Omega)}) \ \\,
\end{eqnarray}
where $C(m_A, M_A, M_\gamma,\|\nabla u\|_{L^{\infty}_{loc}(\Omega)})$ is a positive constant.

So by \eqref{eq:4-6-8} and \eqref{zio} we get:
\begin{equation}\label{ziocarlo}
     -(I_{4} + I_{6} + I_{8}) \le \theta  \int_{\Omega} \frac{A(H(\nabla u))|\nabla u_i|^{2}G_\varepsilon (u_i)}{|x-y|^{\gamma}|u_i|^{\beta}u_i} \varphi^{2}  \ dx + C,
\end{equation}
where $C(c_2,K_1,K_2,m_a, M_A,M_\gamma,\gamma,\theta,\|\nabla u\|_{L^{\infty}_{loc}(\Omega)})$ is a positive constant.

In a similar way we can estimate $I_5+I_7+I_9$.
\begin{eqnarray}\label{eq:5-7-9}
      \nonumber  -(I_{5} + I_{7} + I_{9}) &=& -2 \int_{\Omega}  \big[(A(H(\nabla u))+ H(\nabla u)A'(H(\nabla u))) \langle \nabla H(\nabla u), \nabla u_i \rangle \langle \nabla H(\nabla u), \nabla \varphi \rangle  \\
    \nonumber &&\qquad + H(\nabla u) A(H(\nabla u)) \langle D^{2}H(\nabla u) \nabla u_i, \nabla \varphi \rangle \big ] T_{\epsilon}(u_i) H_{\delta} \varphi  \ dx \\ 
   &\le& \tilde{C}( c_2,K_1,K_2,M_{A}) \int_{\Omega} A(H(\nabla u)) |\nabla u_i| |\nabla \varphi| |T_{\epsilon}(u_i)| H_{\delta} \varphi  \ dx \\ 
\nonumber &\le& \theta  \int_{\Omega} \frac{A(H(\nabla u))|\nabla u_i|^{2}G_\varepsilon (u_i)}{|x-y|^{\gamma}|u_i|^{\beta}u_i} \varphi^{2} + \bar{C},
\end{eqnarray}

where $\bar{C}(c_2,K_1,K_2,m_A, M_A,M_\gamma,\gamma,\theta,\|\nabla u\|_{L^{\infty}_{loc}(\Omega)})$ is a positive constant.

Finally we estimate the term $I_{10}$. By definition of $T_{\epsilon}$ and $H_\delta$ (see \eqref{eq:Hdelta}), we have 
\begin{equation}\label{eq:I_10}
    \begin{split}
         I_{10} &\le \int_{\Omega} \ |f_i| |T_{\epsilon}(u_i)| H_{\delta} \varphi^{2}  \ dx \le \int_{\Omega} \frac{|f_i| |T_{\epsilon}(u_i)|}{|x-y|^{\gamma}} \varphi^{2}  \ dx  \\
        &\le \int_{\Omega} \frac{|f_i| |u_i|^{1-\beta}}{|x-y|^{\gamma}} \varphi^{2}  \ dx  \le \bar{C},
    \end{split}
\end{equation}
   where $\bar{C}(M_\gamma , \|f\|_{1,\frac{N}{(N-\gamma)^-}} \, \|\nabla u\|_{L^{\infty}_{loc}(\Omega)})$ is a positive constant.

By \eqref{eq:sx}, \eqref{ziocarlo}, \eqref{eq:5-7-9}, \eqref{eq:I_10}  we obtain:
\begin{eqnarray*}
   \Tilde{C}(c_1,m_{A},\Lambda) \int_{\Omega} A(H(\nabla u)) |\nabla u_i|^{2} T'_{\epsilon}(u_i) H_{\delta} \varphi^{2} \ dx 
\leq  \theta \int_{\Omega} \frac{A(H(\nabla u))|\nabla u_i|^{2}G_\varepsilon (u_i)}{|x-y|^{\gamma}|u_i|^{\beta}u_i} \varphi^{2}+C.
\end{eqnarray*}
Passing $\delta\to 0$,  by Fatou Lemma, we obtain:
\begin{equation}
  \Tilde{C}(m_{A},c_1,\Lambda) \int_{\Omega} \frac{A(H(\nabla u)) |\nabla u_i|^{2}}{|x-y|^{\gamma}|u_i|^{\beta}} \biggl( G'_{\epsilon}(u_i) -(\beta+ \theta) \frac{G_{\epsilon}(u_i)}{u_i} \biggr) \varphi^{2}  \ dx  \le C.
\end{equation} 
    
Now choosing $\theta$ sufficiently small such that $1-\beta-\theta>0$, since for $\epsilon \to 0$,
\begin{equation*}
   \biggl( G'_{\epsilon}(u_i) -(\beta+ \theta) \frac{G_{\epsilon}(u_i)}{u_i} \biggr) \to 1-\beta-\theta,
\end{equation*}
by Fatou Lemma,  we get
\begin{equation}
    \int_{B_{R}(x_0)\setminus Z_{u}} \frac{A(H(\nabla u)) |\nabla u_i|^{2}}{|x-y|^{\gamma} |u_i|^{\beta}}  \ dx    \le C,
\end{equation}
where $C := C(c_1,c_2, K_1,K_2, m_{A}, M_{A},x_0, R, \theta, \gamma, \beta, \Lambda, \|\nabla u\|_{L^{\infty}_{loc}(\Omega)},\|f\|_{1,\frac{N}{(N-\gamma)^-}}). $ 
\end{proof}

As a consequence of result above we get information about the regularity of the stress field of \eqref{eq:forte}. To make readable the proof of Theorem \ref{eq:campovettoriale}, we consider each statement of it as an autonomous proposition that next we are going to prove.

\begin{prop}\label{eq:campovettoriale1}
    Let $u\in W_{loc}^{1,\mathcal A}(\Omega)$ be a weak solution of \eqref{eq:forte}, where $f$ satisfies $(H_f)$. \newline 
  If $0 < m_{A} < M_{A},$ then 
    \begin{equation} \label{eq:ma posit}
        A(H(\nabla u))^{\frac{k-1}{m_{A}}} \, H(\nabla u)\nabla H(\nabla u)\in W^{1,2}_{loc}(\Omega,\R^{N}),
    \end{equation}
    for any $\displaystyle k >1+ \frac{m_{A}}{2} - \frac{m_{A}}{2M_{A}}.$ \newline
\end{prop}

\begin{proof}
    For $\epsilon > 0$ define
    \begin{equation} \label{eq:stress field}
        V_{\epsilon,j} := A(H(\nabla u))^{\frac{k-1}{m_{A}}} H(\nabla u)H_{\eta_{j}}(\nabla u) J_{\epsilon}(|\nabla u |),
    \end{equation}
    where $\displaystyle J_{\epsilon}(t) := \frac{G_{\epsilon}(t)}{t}$, and $J_\epsilon(0)=0$, $\displaystyle H_{\eta_{j}} := \frac{\partial H}{\partial \eta_{j}}$ and $H(\eta) = H(\eta_1,...,\eta_n)$. Since $ J_{\epsilon}(|\nabla u|)=0$ whenever $|\nabla u|\leq \epsilon$,  we have
    \begin{eqnarray}\label{eq:gradstress}
       \nonumber  \frac{\partial V_{\epsilon,j}}{\partial x_i} &=& \left( \frac{k-1}{m_A} \right) A(H(\nabla u))^{\frac{k-1}{m_A}-1} A^{'}(H(\nabla u)) H(\nabla u) \langle \nabla H(\nabla u),\nabla u_i \rangle J_{\epsilon}(|\nabla u|) H_{\eta_{j}}(\nabla u)  \\
        &&+ A(H(\nabla u))^{\frac{k-1}{m_A}} \langle \nabla H(\nabla u),\nabla u_i \rangle J_{\epsilon}(|\nabla u|) H_{\eta_{j}}(\nabla u) \\
       \nonumber  &&+ A(H(\nabla u))^{\frac{k-1}{m_A}} H(\nabla u) J^{'}_{\epsilon}(|\nabla u|) \partial_{x_i} \left(|\nabla u|\right) H_{\eta_j}(\nabla u)\\
        \nonumber &&+ A(H(\nabla u))^{\frac{k-1}{m_A}} H(\nabla u) J_{\epsilon}(|\nabla u|) \langle \nabla H_{\eta_j}(\nabla u), \nabla u_i\rangle,
    \end{eqnarray}
whenever $|\nabla u|>\epsilon$.
    Using \eqref{eq:hp principale}, \eqref{grad H limitato}, \eqref{eq:Hess H limitato} and the definition of $J_{\epsilon}$ we get:
    \begin{equation}\label{eq:grad}
    \begin{split}
        \left| \frac{ \partial V_{\epsilon,j}}{\partial x_i} \right| &\le C(k,K_1, m_A,M_A) A(H(\nabla u))^{\frac{k-1}{m_A}} |\nabla u_i| \chi_{\{|\nabla u|>\epsilon\}} \\
        &\quad +K_1A(H(\nabla u))^{\frac{k-1}{m_A}} H(\nabla u) |J^{'}_{\epsilon}(|\nabla u|)| |\nabla u_i|\\
        &\quad + C(c_2,K_2) A(H(\nabla u))^{\frac{k-1}{m_A}} |\nabla u_i|\chi_{\{|\nabla u|>\epsilon\}}\le C A(H(\nabla u))^{\frac{k-1}{m_A}} |\nabla u_i|\chi_{\{|\nabla u|>\epsilon\}} ,
    \end{split}
    \end{equation}
    where $C=C(c_2, k,K_1, K_2, m_A,M_A)$ is a positive constant. Indeed we note that
    \begin{eqnarray*}
    	A(H(\nabla u))^{\frac{k-1}{m_A}} H(\nabla u) |J^{'}_{\epsilon}(|\nabla u|)| |\nabla u_i|=0, \ a.e.
    \end{eqnarray*}
    if $|\nabla u|\in [0,\epsilon]\cup[2\epsilon,+\infty)$. Whenever $|\nabla u|\in (\epsilon,2\epsilon)$, since $\displaystyle J'(t)=\frac{2\epsilon}{t^2}$ and using \eqref{H equiv euclidea} we have
    \begin{eqnarray*}
    	&&K_1 A(H(\nabla u))^{\frac{k-1}{m_A}} H(\nabla u) |J^{'}_{\epsilon}(|\nabla u|)| |\nabla u_i|\leq c_2K_1A(H(\nabla u))^{\frac{k-1}{m_A}}|\nabla u|\frac{2\epsilon}{|\nabla u|^2}|\nabla u_i|\\
    	&&\mbox{(by $|\nabla u|>\epsilon$)}\leq 2c_2K_1A(H(\nabla u))^{\frac{k-1}{m_A}}|\nabla u_i|\leq  C(c_2,K_1)A(H(\nabla u))^{\frac{k-1}{m_A}}|\nabla u_i|\chi_{\{|\nabla u|>\epsilon\}}. 
    \end{eqnarray*} 
Therefore
$$
 K_1A(H(\nabla u))^{\frac{k-1}{m_A}} H(\nabla u) |J^{'}_{\epsilon}(|\nabla u|)| |\nabla u_i|\leq  C(c_2,K_1)A(H(\nabla u))^{\frac{k-1}{m_A}}|\nabla u_i|\chi_{\{|\nabla u|>\epsilon\}}, 
$$
and \eqref{eq:grad} holds. \\   
    Fixed $x_0 \in \Omega$ and $R>0$ such that $B := B_{R}(x_0) \subset \subset \Omega$, using  \eqref{eq:grad} and \eqref{eq:hp cianchi} we get
    \begin{eqnarray}
     \nonumber  \int_{B} |\nabla V_{\epsilon,j}|^{2} \, dx &\le& C(c_2, k, K_1, K_2, m_A,M_A) \int_{B} A(H(\nabla u))^{\frac{2(k-1)}{m_{A}}}|D^{2}u|^{2}\chi_{\{|\nabla u|>\epsilon\}}   \, dx \\
        &\le& C \int_{B} A(H(\nabla u))^{\frac{2(k-1)}{m_{A}} -1} \frac{A(H(\nabla u)) |D^{2}u|^{2}}{|u_i|^{\beta}} |\nabla u|^{\beta} \chi_{\{|\nabla u|>\epsilon\}} \, dx \\
        \nonumber &\le& C \int_{B} A(H(\nabla u))^{\frac{2(k-1)}{m_{A}} -1} \frac{A(H(\nabla u)) |D^{2}u|^{2}}{|u_i|^{\beta}} H(\nabla u)^{\beta}\chi_{\{|\nabla u|>\epsilon\}}\, dx \\
      \nonumber  &\le& C\int_{B \cap H_{>}} A(H(\nabla u))^{\frac{2(k-1)}{m_{A}} -1} H(\nabla u)^{\beta} \frac{A(H(\nabla u)) |D^{2}u|^{2}}{|u_i|^{\beta}} \ \, dx \\
        \nonumber &&+ C\int_{B \cap H_{<}} A(H(\nabla u))^{\frac{2(k-1)}{m_{A}} -1} A(H(\nabla u))^{\frac{\beta}{M_A}}\frac{A(H(\nabla u)) |D^{2}u|^{2}}{|u_i|^{\beta}}\chi_{\{|\nabla u|>\epsilon\}}  \, dx. 
    \end{eqnarray}
    where $H_{<}:=\{H(\nabla u) \le 1\}$, $H_{>}:=\{H(\nabla u) \ge 1\}$ and $C=C(c_1,c_2,k, K_1,K_2, m_A,M_A)>0$.\\
        Now, since $u\in C^2(\Omega\setminus Z_u)$ and by Theorem \ref{stime derivate seconde}, choosing $\beta \sim 1^{-},$ for any $k >  1+\frac{m_A}{2} - \frac{m_A}{2M_A}$ we get
    \begin{equation}\label{3.47}
        \begin{split}
            \int_{B} |\nabla V_{\epsilon,j}|^{2} &\le C+  C \int_{{(B \cap H_{<})\setminus Z_u}} \frac{A(H(\nabla u)) |D^{2}u|^{2}}{|u_i|^{\beta}}  \, dx \\
            &\le   C(c_1,c_2, \hat C, k, K_1,K_2, m_A,M_A,\|\nabla u\|_{L^{\infty}_{loc}(\Omega)}), \\
        \end{split}
    \end{equation}
    where $C(c_1,c_2,\hat C, k, K_1,K_2, m_A,M_A,\|\nabla u\|_{L^{\infty}_{loc}(\Omega)})$ is a positive constant and $\Hat{C}$ is given by Theorem~\ref{stime derivate seconde}. \\
    Since $W^{1,2}_{loc}(\Omega)$ is a reflexive space, there exists $\Tilde{V}_j \in W^{1,2}_{loc}(\Omega)$ such that
    \begin{equation}\label{convdeb}
        V_{\epsilon,j} \rightharpoonup \Tilde{V}_j, \quad \text{for  } \epsilon \rightarrow 0.
    \end{equation}
    By the compact embedding, $V_{\epsilon,j} \rightarrow \Tilde{V}_j$ in $L^{q}(\Omega)$ with $q<2^{*}$ and up to subsequence $V_{\epsilon,j} \rightarrow \Tilde{V}_j$ a.e. in $\Omega$. 
    Because of the choice of $k$, it is easy to verify that $V_{\epsilon,j} \rightarrow  A(H(\nabla u))^{\frac{k-1}{m_{A}}} \, H(\nabla u)H_{\eta_{j}}(\nabla u)$ a.e. in $\Omega$, then $$\tilde V_j= A(H(\nabla u))^{\frac{k-1}{m_{A}}} \, H(\nabla u)H_{\eta_{j}}(\nabla u)  \in W^{1,2}_{loc}(\Omega).$$
\end{proof}
\begin{rem}
 The choice $k=m_A+1$ allows us to recover the optimal result in \cite{ACCFM} obtained for $f\in L^2$. It is worth noting that this choice will be applicable in all subsequent propositions, therefore, in every case, we encompass the main result in \cite{ACCFM}.
\end{rem}
\begin{prop}\label{eq:campovettoriale2}
    Let $u\in W_{loc}^{1,\mathcal A}(\Omega)$ be a weak solution of \eqref{eq:forte} and $f$ satisfying $(H_f)$. \newline 
   If $m_{A} < M_{A} < 0,$ then 
    \begin{equation}\label{eq:Ma negat}
        A(H(\nabla u))^{\frac{k-1}{m_{A}}} \, H(\nabla u)\nabla H(\nabla u)\in W^{1,2}_{loc}(\Omega,\R^{N})
    \end{equation}
    for any $\displaystyle k > \frac 12 +\frac{m_{A}}{2}.$ \newline
\end{prop}
\begin{proof}
   Consider $V_{\epsilon,j}$ in \eqref{eq:stress field}. We only prove that
   $$
      \int_{B} |\nabla V_{\epsilon,j}|^{2} \, dx <+\infty  
    $$
    Indeed the thesis follows exploiting arguments in Proposition \ref{eq:campovettoriale1} . Using \eqref{eq:grad} we already know that
    \begin{equation}\label{help}
    \begin{split}
        \int_{B} |\nabla V_{\epsilon,j}|^{2} \, dx &\le C\int_{B} A(H(\nabla u))^{\frac{2(k-1)-m_A}{m_{A}}} \frac{A(H(\nabla u)) |D^{2}u|^{2}}{|u_i|^{\beta}} |\nabla u|^{\beta}\chi_{\{|\nabla u|>\epsilon\}}  \, dx. \\
    \end{split}
    \end{equation}
    Now let us consider the two cases $2(k-1)-m_A \ge 0$ and $2(k-1)-m_A \le 0.$ 
    
    In the first, i.e. $k \ge \frac{m_A+2}{2}$, using 
\eqref{eq:hp cianchi} and $u\in C^2(\Omega\setminus Z_u)$, from Theorem \ref{stime derivate seconde} we get
    \begin{equation}\label{basta1}
        \begin{split}
            \int_{B} |\nabla V_{\epsilon,j}|^{2} \, dx &\le C \int_{B \cap H_{<}} A(H(\nabla u))^{\frac{2(k-1)-m_A}{m_{A}}} |\nabla u|^{\beta}  \frac{A(H(\nabla u)) |D^{2}u|^{2}}{|u_i|^{\beta}}\chi_{\{|\nabla u|>\epsilon\}}  \, dx \\
            &\quad + C\int_{B \cap H_{>}} A(H(\nabla u))^{\frac{2(k-1)-m_A}{m_{A}}} |\nabla u|^{\beta}  \frac{A(H(\nabla u)) |D^{2}u|^{2}}{|u_i|^{\beta}} \, dx \\
            &\le C \int_{B \cap H_{<}} H(\nabla u)^{(2(k-1)-m_A)\frac{M_A}{m_A}+\beta}\frac{A(H(\nabla u)) |D^{2}u|^{2}}{|u_i|^{\beta}} \chi_{\{|\nabla u|>\epsilon\}} \, dx+C<C
                      \end{split}
    \end{equation}
 renaming $C=C(c_1,c_2,k,K_1,K_2,m_A,M_A,\hat C,\|\nabla u\|_{L^{\infty}_{loc}(\Omega)})$ a positive constant, $\Hat{C}$ is given by Theorem ~\ref{stime derivate seconde} and 
 $$H_{<}:=\{H(\nabla u) \le 1\}, \qquad H_{>}:=\{H(\nabla u) \ge 1\}.$$
    Let us now consider $2(k-1)-m_A \le 0,$ that is $k \le \frac{m_A+2}{2}.$ 
    As in the previous case,  using \eqref{help},
\eqref{eq:hp cianchi} and since $u\in C^2(\Omega\setminus Z_u)$ we obtain
    \begin{equation}
        \begin{split}
             \int_{B} |\nabla V_{\epsilon,j}|^{2} \, dx &\le C \int_{B \cap H_{<}} A(H(\nabla u))^{\frac{2(k-1)-m_A}{m_{A}}} |\nabla u|^{\beta}  \frac{A(H(\nabla u)) |D^{2}u|^{2}}{|u_i|^{\beta}}\chi_{\{|\nabla u|>\epsilon\}}  \, dx \\
            &\quad +  C\int_{B \cap H_{>}} A(H(\nabla u))^{\frac{2(k-1)-m_A}{m_{A}}} |\nabla u|^{\beta}  \frac{A(H(\nabla u)) |D^{2}u|^{2}}{|u_i|^{\beta}} \, dx \\
            &\le C \int_{B \cap H_{<}} H(\nabla u)^{2(k-1)-m_A} |\nabla u|^{\beta} \frac{A(H(\nabla u)) |D^{2}u|^{2}}{|u_i|^{\beta}} \chi_{\{|\nabla u|>\epsilon\}} \,dx +C\\
            &\le C \int_{(B \cap H_{<})\setminus Z_u} |\nabla u|^{2(k-1)-m_A+\beta} \frac{A(H(\nabla u)) |D^{2}u|^{2}}{|u_i|^{\beta}} \, dx+C.
        \end{split}
    \end{equation} 
  From Theorem \ref{stime derivate seconde}, for $\beta$ close to $1$, 
     we get 
    \begin{equation}\label{basta}
        \int_{B} |\nabla V_{\epsilon,j}|^{2} \, dx \le C(c_2, k, K_1, K_2, m_A,M_A, \Hat{C},\|\nabla u\|_{L^{\infty}_{loc}(\Omega)})
    \end{equation}
    for any $k \in \left(\frac{m_A+1}{2},\frac{m_A+2}{2}\right]$. 
    Exploiting  \eqref{basta1} and \eqref{basta} the claim holds for any $k > \frac{m_A+1}{2}.$  
\end{proof}

    \begin{prop}\label{eq:campovettoriale3}
    Let $u\in W_{loc}^{1,\mathcal A}(\Omega)$ be a weak solution of \eqref{eq:forte}, with $f$ satisfying $(H_f)$. \newline 
 If $m_{A} < 0 < M_{A},$ then
    \begin{equation}\label{eq:discordi}
        A(H(\nabla u))^{\frac{k-1}{M_{A}}} \, H(\nabla u)\nabla H(\nabla u)\in W^{1,2}_{loc}(\Omega,\R^{N}),
    \end{equation}
    for any $\displaystyle k \in \left(\frac12+\frac{M_A}{2}, 1 +\frac {M_A}{2}-\frac{M_A}{2m_A}\right).$  
\end{prop}
\begin{proof}    
  Here we set \begin{equation}\label{eq:stress2}
        V_{\epsilon,j} := A(H(\nabla u))^{\frac{k-1}{M_{A}}} H(\nabla u) J_{\epsilon}(|\nabla u |) H_{\eta_{j}}(\nabla u).
    \end{equation}
    By the same computations of the previous two cases we get
    \begin{equation}\label{eq:bla}
    \begin{split}
        \int_{B} |\nabla V_{\epsilon,j}|^{2} \, dx &\le C(c_2, k, K_1,K_2, M_A) \int_{B \cap H_<} A(H(\nabla u))^{\frac{2(k-1)-M_A}{M_{A}}}  A(H(\nabla u)) |D^{2}u|^{2} \chi_{\{|\nabla u|>\epsilon\}} \, dx \\
        &\quad + C\int_{B \cap H_>} A(H(\nabla u))^{\frac{2(k-1)-M_A}{M_{A}}}   A(H(\nabla u)) |D^{2}u|^{2} \, dx \\
        &=: I_1 + I_2.
    \end{split}
    \end{equation}
 The integral $I_2$ is estimable exploiting that $u\in C^1(\Omega)\cap  C^2(\Omega \setminus Z_u)$. \ \\
 Hence we only explicitly estimate $I_1$.  Suppose that $2(k-1)-M_A\ge 0$. By \eqref{eq:hp cianchi} and from Theorem \ref{stime derivate seconde}, for $\beta$ close to $1$, we have
\begin{equation}\label{eq:blabla12}
\begin{split}
    I_1 &\le C\int_{B \cap H_<} A(H(\nabla u))^{\frac{2(k-1)-M_A}{M_{A}}} |\nabla u|^{\beta} \frac{A(H(\nabla u)) |D^{2}u|^{2}}{|u_i|^{\beta}}\chi_{\{|\nabla u|>\epsilon\}}  \, dx \\
    &\le  C\int_{(B \cap H_<)\setminus Z_u} H(\nabla u)^{\beta+\frac{2(k-1)-M_A}{M_{A}}m_A}  \frac{A(H(\nabla u)) |D^{2}u|^{2}}{|u_i|^{\beta}} \ dx \\
    &\le C(c_1,c_2,k,K_1,K_2, M_A,\hat C,\|\nabla u\|_{L^{\infty}_{loc}(\Omega)}),
\end{split}
\end{equation}
for any $\displaystyle k\in \left[1+\frac{M_A}{2},1+\frac{M_A}{2}-\frac{M_A}{2m_A}\right)$, 
where $C(c_1,c_2,\hat C, k,K_1,K_2, M_A,,\|\nabla u\|_{L^{\infty}_{loc}(\Omega)})$ is a positive constant and $\hat C$ is given by Theorem \ref{stime derivate seconde}.

On the other hand, if $2(k-1)-M_A\le 0$, by \eqref{eq:hp cianchi} and from Theorem \ref{stime derivate seconde}, with $\beta$ close to $1$, we obtain
\begin{equation}\label{eq:blabla}
\begin{split}
    I_1 &\le C\int_{B \cap H_<} A(H(\nabla u))^{\frac{2(k-1)-M_A}{M_{A}}} |\nabla u|^{\beta} \frac{A(H(\nabla u)) |D^{2}u|^{2}}{|u_i|^{\beta}}\chi_{\{|\nabla u|>\epsilon\}}  \, dx \\
    &\le  C \int_{(B \cap H_<)\setminus Z_u} H(\nabla u)^{\beta+2(k-1)-M_A}  \frac{A(H(\nabla u)) |D^{2}u|^{2}}{|u_i|^{\beta}} \ dx \\
    &\le C(c_1,c_2,k,\hat C,K_1,K_2,M_A,\|\nabla u\|_{L^{\infty}_{loc}(\Omega)}),
\end{split}
\end{equation}
for any $k\displaystyle \in \left(\frac 12+\frac{M_A}{2},1+\frac{M_A}{2}\right]$, 
where $C(c_1,c_2,k,\hat C,K_1,K_2,M_A,\|\nabla u\|_{L^{\infty}_{loc}(\Omega)}),$ is a positive constant. 
By \eqref{eq:bla}, \eqref{eq:blabla12} and \eqref{eq:blabla} we get
\begin{equation}
    \begin{split}
        \int_{B} |\nabla V_{\epsilon,j}|^{2} \, dx \le C(c_1,c_2,k,\hat C,K_1,K_2,M_A,\|\nabla u\|_{L^{\infty}_{loc}(\Omega)}).
    \end{split}
\end{equation}

The thesis \eqref{eq:discordi} follows as in the previous cases.
\end{proof}
\begin{prop}\label{eq:campovettoriale4}
    Let $u\in W_{loc}^{1,\mathcal A}(\Omega)$ be a weak solution of \eqref{eq:forte} and $f$ satisfies $(H_f)$. \newline 
   If $m_{A} = 0$ and $M_{A} > 0$, then
    \begin{equation}\label{eq:unonullo}
        A(H(\nabla u))^{k-1} \, H(\nabla u)\nabla H(\nabla u) \in W^{1,2}_{loc}(\Omega,\R^{N})
    \end{equation}
    for any $\displaystyle k > \frac{3}{2}-\frac{1}{2M_A}.$
\end{prop}
\begin{proof}Now we prove the case \eqref{eq:unonullo}. Let us consider 
\begin{equation}\label{cmpunozero}
V_{\epsilon,j} := A(H(\nabla u))^{k-1} H(\nabla u) J_{\epsilon}(|\nabla u |) H_{\eta_{j}}(\nabla u).
\end{equation}
As we did before, we distinguish two cases. If $2(k-1)-1\ge 0$, using \eqref{eq:hp cianchi} and Theorem \ref{stime derivate seconde} we can deduce
\begin{equation}
    \begin{split}
        &\int_{B} |\nabla V_{\epsilon,j}|^{2} \ dx \le C(c_2, k,K_1,K_2,M_A) \int_{B} A(H(\nabla u))^{2(k-1)} |D^{2}u|^{2}\chi_{\{|\nabla u|>\epsilon\}}  \ dx \\
        &\le C\int_{B} A(H(\nabla u))^{2(k-1)-1} H(\nabla u)^{\beta} \frac{A(H(\nabla u)) |D^2u|^{2}}{|u_i|^{\beta}} \chi_{\{|\nabla u|>\epsilon\}} \ dx \\
        &\le C
        \int_{B \cap H_<} A(H(\nabla u))^{2(k-1)-1} H(\nabla u)^{\beta} \frac{A(H(\nabla u)) |D^2u|^{2}}{|u_i|^{\beta}}\chi_{\{|\nabla u|>\epsilon\}}  \ dx \\
        &\quad + C   \int_{B \cap H_>} A(H(\nabla u))^{2(k-1)-1} H(\nabla u)^{\beta} \frac{A(H(\nabla u)) |D^2u|^{2}}{|u_i|^{\beta}} \ dx \\
        &\le C    \int_{(B \cap H_<)\setminus Z_u} \frac{A(H(\nabla u)) |D^2u|^{2}}{|u_i|^{\beta}} \ dx +C\\
        &\le C(c_1,c_2,\hat C, k,K_1,K_2,M_A,\|\nabla u\|_{L^{\infty}_{loc}(\Omega)})
    \end{split}
\end{equation}
where $C(c_1,c_2,\hat C, k,K_1,K_2,M_A,\|\nabla u\|_{L^{\infty}_{loc}(\Omega)})$ is a positive constant.

If $2(k-1)-1\le 0$, by \eqref{eq:hp cianchi} and Theorem \ref{stime derivate seconde}, for $\beta$ close to $1$, and for any $k> \frac 32-\frac{1}{2M_A}$, we can deduce
\begin{equation}
    \begin{split}
        \int_{B} |\nabla V_{\epsilon,j}|^{2} \ dx &\le C(c_1,c_2,k,K_1, K_2, M_A)  \int_{B} A(H(\nabla u))^{2(k-1)} |D^{2}u|^{2} \ dx \\
        &\le C\int_{B} A(H(\nabla u))^{2(k-1)-1} H(\nabla u)^{\beta} \frac{A(H(\nabla u)) |D^2u|^{2}}{|u_i|^{\beta}}\chi_{\{|\nabla u|>\epsilon\}}  \ dx \\
        &\le C
        \int_{B \cap H_<} A(H(\nabla u))^{2(k-1)-1} H(\nabla u)^{\beta} \frac{A(H(\nabla u)) |D^2u|^{2}}{|u_i|^{\beta}}\chi_{\{|\nabla u|>\epsilon\}}  \ dx \\
        &\quad + C
        \int_{B \cap H_>} A(H(\nabla u))^{2(k-1)-1} H(\nabla u)^{\beta} \frac{A(H(\nabla u)) |D^2u|^{2}}{|u_i|^{\beta}} \ dx \\
        &\le C
        \int_{B \cap H_<} H(\nabla u)^{\beta+(2(k-1)-1)M_A} \frac{A(H(\nabla u)) |D^2u|^{2}}{|u_i|^{\beta}} \chi_{\{|\nabla u|>\epsilon\}} \ dx + C\\
       &\le C(c_1, c_2, \hat C, k, K_1, K_2, M_A,\|\nabla u\|_{L^{\infty}_{loc}(\Omega)})
    \end{split}
\end{equation}
where $C(c_1, c_2, \hat C, k, K_1, K_2, M_A,\|\nabla u\|_{L^{\infty}_{loc}(\Omega)})$ is a positive constant. 
\end{proof}
\begin{prop}
   Let $u\in W_{loc}^{1,\mathcal A}(\Omega)$ be a weak solution of \eqref{eq:forte} and $f$ satisfies $(H_f)$.\newline 
   If $m_{A} < 0$ and $M_{A} = 0$, then 
    \begin{equation}\label{eq:unonullo2}
        A(H(\nabla u))^{k-1} \, H(\nabla u)\nabla H(\nabla u) \in W^{1,2}_{loc}(\Omega,\R^{N})
    \end{equation} for any $\displaystyle k<\frac{3}{2}-\frac{1}{2m_A}.$
\end{prop}    
\begin{proof}
The proof of \eqref{eq:unonullo2} is similar  to the case \eqref{eq:unonullo}.
 We take  $V_{\epsilon,j}$ as in \eqref{cmpunozero}. Let us consider first the case $2(k-1)-1 \ge 0.$ By Theorem  \ref{stime derivate seconde}, by \eqref{eq:hp cianchi} and choosing $\beta \sim 1^{-}$ we get
\begin{equation}\label{eq:ge}
\begin{split}
    \int_{B} |\nabla V_{\epsilon,j}|^{2} \ dx &\le C(c_2, k, K_1, K_2, M_A) \int_{B} A(H(\nabla u))^{2(k-1)} |D^{2}u|^{2} \ dx \\
    &\le C
    \int_{B} A(H(\nabla u))^{2(k-1)-1} H(\nabla u)^{\beta}  \frac{A(H(\nabla u)) |D^2u|^{2}}{|u_i|^{\beta}}\chi_{\{|\nabla u|>\epsilon\}}  dx \\
    &\le C
    \int_{B \cap H_<} A(H(\nabla u))^{2(k-1)-1} H(\nabla u)^{\beta}  \frac{A(H(\nabla u)) |D^2u|^{2}}{|u_i|^{\beta}}\chi_{\{|\nabla u|>\epsilon\}} dx \\
    &\quad + C\int_{B \cap H_>} A(H(\nabla u))^{2(k-1)-1} H(\nabla u)^{\beta}  \frac{A(H(\nabla u)) |D^2u|^{2}}{|u_i|^{\beta}}\ dx \\
    &\le C
    \int_{(B \cap H_<)\setminus Z_u}  H(\nabla u)^{\beta + (2(k-1)-1)m_{A}}  \frac{A(H(\nabla u)) |D^2u|^{2}}{|u_i|^{\beta}} \, dx \\
    &\le C(c_1, c_2, \Hat{C}, k, K_1, K_2, M_A, \|\nabla u\|_{L^{\infty}_{loc}(\Omega)})
\end{split}
\end{equation}
for any $k\in \left[ \frac{3}{2}, \frac{3}{2} - \frac{1}{2m_{A}} \right),$ where $C(c_1, c_2, \Hat{C}, k, K_1, K_2, M_A, \|\nabla u\|_{L^{\infty}_{loc}(\Omega)})$ is positive constant and $\Hat{C}$ is given by Theorem \ref{stime derivate seconde}. 

If $2(k-1)-1 < 0,$ by Theorem \ref{stime derivate seconde} and by \eqref{eq:hp cianchi} we get 
\begin{equation}\label{eq:le}
    \begin{split}
    \int_{B} |\nabla V_{\epsilon,j}|^{2} \ dx &\le C(c_2, k, K_1, K_2, M_A) \int_{B} A(H(\nabla u))^{2(k-1)} |D^{2}u|^{2} \ dx \\
    &\le C
    \int_{B} A(H(\nabla u))^{2(k-1)-1} H(\nabla u)^{\beta}  \frac{A(H(\nabla u)) |D^2u|^{2}}{|u_i|^{\beta}}\chi_{\{|\nabla u|>\epsilon\}}  dx \\
    &\le C
    \int_{B \cap H_<} A(H(\nabla u))^{2(k-1)-1} H(\nabla u)^{\beta}  \frac{A(H(\nabla u)) |D^2u|^{2}}{|u_i|^{\beta}}\chi_{\{|\nabla u|>\epsilon\}} dx \\
    &\quad + C \int_{B \cap H_>} A(H(\nabla u))^{2(k-1)-1} H(\nabla u)^{\beta}  \frac{A(H(\nabla u)) |D^2u|^{2}}{|u_i|^{\beta}}\ dx \\
    &\le C(c_1,c_2, \Hat{C}, k, K_1,K_2, M_A,\|\nabla u\|_{L^{\infty}_{loc}(\Omega)})
    \end{split}
\end{equation}
for any $k < \frac{3}{2},$ where $C(c_1,c_2, \Hat{C}, k, K_1,K_2, M_A,\|\nabla u\|_{L^{\infty}_{loc}(\Omega)})$ is a positive constant. Therefore by \eqref{eq:ge} and \eqref{eq:le}, proceeding as in the previous cases, \eqref{eq:unonullo2} holds for any $k < \frac{3}{2} - \frac{1}{2m_{A}}.$
\end{proof}

\section{Second-order regularity of weak solutions.}\label{sezz4}

We conclude our paper proving our main result. In order to do this, we first prove an integrability property for $A(H(\nabla u))^{-1}$. 

\begin{thm}\label{inverso del peso}
Let $u\in W_{loc}^{1,\mathcal A}(\Omega)$ be a weak solution of \eqref{eq:forte}, with $f$ satisfying $(H_f)$ and  $$f(x) \ge c(x_0,R) > 0 \quad\text{in } B_{2R}(x_0),$$ where  $B_{2R}(x_0) \subset \subset \Omega,$  $x_0 \in \Omega$ and $R>0$. Consider $y \in \Omega.$    
    
    Then, we have 
    \begin{equation}
        \int_{B_{R}(x_0)} \frac{1}{A(H(\nabla u))^{\alpha r}} \frac{1}{|x-y|^{\gamma}} \ dx  \le C
    \end{equation}
    where $ \alpha := \frac{M_A +1}{M_A}$, $M_A >0,$ $r <1$, $\gamma < N-2$ if $N \ge 3$ while $\gamma = 0$ if $N = 2$ and \newline $C = C(K_1, m_{A},M_{A}, r, R, x_0, \gamma,\sigma,\|\nabla u\|_{L^{\infty}_{loc}(\Omega)})$.
\end{thm}

\begin{proof} 
     For $n \in \N$ consider $G_{\frac{1}{n}}(t)$ following \eqref{eq:Geps}. 
 Fix an arbitrary $\epsilon > 0 $, let $n_{\epsilon} \in \N$ such that for any $n > n_{\epsilon}$ it holds
  \begin{equation}\label{eq:uniforme}
    \sup_{t\geq0} |G_{\frac1n}(t) - t| < \epsilon.
 \end{equation}
Choosing $N>n_\varepsilon$, consider the following test function:
    \begin{equation}\label{eq:testpeso}
        \varphi = \frac{ H_{\delta}(|x-y|) \psi^{2}}{(G_{\frac1N}(A(H(\nabla u)))+\epsilon)^{\alpha r}},
    \end{equation}
    where $\psi$ and $H_{\delta}$ are defined in (\ref{eq:varphi}), (\ref{eq:Hdelta}).
    Then
    \begin{equation}
    \begin{split}
        \nabla \varphi = \frac{2\psi H_{\delta} \nabla \psi }{(G_{\frac1N}(A(H(\nabla u)))+\epsilon)^{\alpha r}} &+ \frac{\psi^2 \nabla H_{\delta}}{(G_{\frac1N}(A(H(\nabla u)))+\epsilon)^{\alpha r}} \\
        &- \alpha r \psi^2 H_{\delta} \frac{G'_{\frac1N}(A(H(\nabla u))) A'(H(\nabla u)) D^2u \nabla H(\nabla u)}{(G_{\frac1N}(A(H(\nabla u)))+\epsilon)^{\alpha r+1}}.
    \end{split}
    \end{equation}
    Since $f(x) \ge C(x_0,R) > 0$ in $B:=B_{2R}(x_0),$ testing $\eqref{eq:testpeso}$ in (\ref{eq:debole}) we get
    \begin{eqnarray}
        &&C(R,x_0) \int_{B} \frac{H_{\delta}(|x-y|) \psi^{2}}{(G_{\frac1N}(A(H(\nabla u)))+\epsilon)^{\alpha r}} \ dx \\ 
        \nonumber  &\le& -\alpha r \int_{B} \frac{ A(H(\nabla u)) H(\nabla u) G'_{\frac1N}(A(H(\nabla u))) A'(H(\nabla u)) \langle \nabla H (\nabla u) , D^{2}u \nabla H(\nabla u) \rangle}{\big(G_{\frac1N}(A(H(\nabla u)))+\epsilon \big)^{\alpha r+1}} H_{\delta} \psi^{2} \ dx  \\
         \nonumber &&+ \int_{B} \frac{ A(H(\nabla u)) H(\nabla u) \langle \nabla H(\nabla u), \nabla H_{\delta} \rangle}{\big(G_{\frac1N}(A(H(\nabla u)))+\epsilon \big)^{\alpha r}} \psi^{2}  \ dx + 2 \int_{\Omega} \frac{ A(H(\nabla u)) H(\nabla u) \langle \nabla H(\nabla u), \nabla \psi \rangle}{\big(G_{\frac1N}(A(H(\nabla u)))+\epsilon \big)^{\alpha r}} H_{\delta}\psi \ dx.
    \end{eqnarray}
    
    Therefore, denoting by $I_F$ the left-hand side of the previous inequality, we have
    \begin{eqnarray}\label{eq:conto}
        \nonumber I_F&:=& \int_{B} \frac{H_{\delta}(|x-y|) \psi^{2}}{(G_{\frac1N}(A(H(\nabla u)))+\epsilon)^{\alpha r}} \ dx \\
       \nonumber  &\le& C(r,R,x_0,\alpha) \int_{B} \frac{|H(\nabla u) A(H(\nabla u))  A'(H(\nabla u))| |D^{2}u|  |\nabla H (\nabla u)|^{2}}{\big(G_{\frac1N}(A(H(\nabla u)))+\epsilon \big)^{\alpha r +1}} H_{\delta} \psi^{2} \ dx  \\
        &&  + C(R,x_0) \int_{B} \frac{H(\nabla u) A(H(\nabla u)) |\nabla H(\nabla u)| |\nabla H_{\delta}|}{\big(G_{\frac1N}(A(H(\nabla u)))+\epsilon \big)^{\alpha r}} \psi^{2} \ dx  \\
        &&\nonumber  + C(R,x_0) \int_{B} \frac{H(\nabla u) A(H(\nabla u))  |\nabla H(\nabla u)| |\nabla \psi| }{\big(G_{\frac1N}(A(H(\nabla u)))+\epsilon \big)^{\alpha r}} H_{\delta}\psi \ dx\\
        \nonumber &=:& I_1 + I_2 + I_3,
    \end{eqnarray}
    where $C(r,R,x_0,\alpha)$ and $C(R,x_0)$ are positive constants.
    First, we estimate the term $I_3$. Using \eqref{eq:hp cianchi}, \eqref{grad H limitato}, \eqref{eq:uniforme} and since $\alpha =(M_A+1)/M_A$ we get:
    \begin{eqnarray}
            \nonumber I_3 &\le& C(K_1,R,x_0) \int_{B} \frac{ A(H(\nabla u)) H(\nabla u) }{\big(G_{\frac1N}(A(H(\nabla u)))+\epsilon \big)^{\alpha r}}  H_{\delta} \psi \ dx \\
            \nonumber &\le& C\int_{B \cap \{A < 1\}} \frac{ A(H(\nabla u)) H(\nabla u) }{\big(G_{\frac1N}(A(H(\nabla u)))+\epsilon \big)^{\alpha r}}  H_{\delta} \psi \ dx \\ 
            && + C\int_{B \cap \{A \ge 1\}} \frac{ A(H(\nabla u)) H(\nabla u) }{\big(G_{\frac1N}(A(H(\nabla u)))+\epsilon \big)^{\alpha r}}  H_{\delta} \psi \ dx\\
           \nonumber &\le& C \int_{B \cap \{A < 1\}} \frac{ A(H(\nabla u))^{r} H(\nabla u) }{\big(G_{\frac1N}(A(H(\nabla u)))+\epsilon \big)^{\alpha r}}  H_{\delta} \psi \ dx + \tilde{C}(K_1,r, R,x_0,\alpha,\gamma,\|\nabla u\|_{L^{\infty}_{loc}(\Omega)})\\
           \nonumber &=& C\int_{B \cap \{A < 1\}} \frac{  H(\nabla u) }{ A(H(\nabla u))  ^{(\alpha -1) r}}  H_{\delta} \psi \ dx + \tilde{C}\\
          \nonumber  &=& C\int_{B \cap \{A < 1\}} \frac{  H(\nabla u) }{ A(H(\nabla u))  ^{\frac{r}{M_A}}}  H_{\delta} \psi \ dx + \tilde{C}\\
            &\le& C \int_{B \cap \{A < 1\}} \frac{  H(\nabla u) H_{\delta} \psi}{ \min \{H(\nabla u)^{r},H(\nabla u)^{\frac{m_{A}r}{M_A}} \} }   \ dx + \tilde{C},\\
    \end{eqnarray}
    Since $r<1$ and $r \, m_{A} < M_A,$ there exists a positive constant $C=C(K_1,r, R,x_0,\alpha,\gamma,\|\nabla u\|_{L^{\infty}_{loc}(\Omega)})$ such that
    \begin{equation}\label{eq:I3peso}
        I_3 \le C(K_1, r, R, x_0,\alpha,\gamma,\|\nabla u\|_{L^{\infty}_{loc}(\Omega)}).
    \end{equation}

    By a similar computation, we get that there exists a positive constant $C$ such that 
    \begin{equation}\label{eq:I2peso}
        I_2 \le C.
    \end{equation}
    Now we estimate $I_1$; by \eqref{eq:hp principale}, \eqref{grad H limitato} and using a weighted Young inequality we get
    \begin{equation}\label{daino}
        \begin{split}
            I_{1} &\le C(K_1,r,R,x_0,\alpha) \int_{B} \frac{|A(H(\nabla u)) H(\nabla u) A'(H(\nabla u))| |D^{2}u|   \psi^{2} H_{\delta} }{\big(G_{\frac1N}(A(H(\nabla u)))+\epsilon \big)^{\alpha r + 1}} \ dx\\
            &\le C(K_1,m_A,M_A,r,R,x_0,\alpha) \int_{B} \frac{A^{2}(H(\nabla u)) |D^2u| H_{\delta} \psi^{2}}{\big(G_{\frac1N}(A(H(\nabla u)))+\epsilon \big)^{\alpha r + 1}} \ dx \\
            &\le C \left[ \theta \int_{B} \frac{\psi^{2} H_{\delta}}{\big(G_{\frac1N}(A(H(\nabla u)))+\epsilon \big)^{\alpha r}} \ dx + \frac{1}{4\theta} \int_{B} \frac{A^{4}(H(\nabla u)) |D^2u|^{2} \psi^{2} H_{\delta}}{\big(G_{\frac1N}(A(H(\nabla u)))+\epsilon \big)^{\alpha r +2}} \ dx \right]\\
            &\le C \left[\theta I_F +  \frac{1}{4\theta} \int_{B} \frac{A^{2}(H(\nabla u)) |D^2u|^{2} \psi^{2} H_{\delta}}{\big(G_{\frac1N}(A(H(\nabla u)))+\epsilon \big)^{\alpha r}} \ dx \right].
        \end{split}
    \end{equation}
    Now, by the definition of $\alpha=(M_A+1)/M_A$ and \eqref{eq:uniforme}, we get
    \begin{equation}\label{daino1}
    \begin{split}
        \int_{B} \frac{A^{2}(H(\nabla u)) |D^2u|^{2} \psi^{2} H_{\delta}}{\big(G_{\frac1N}(A(H(\nabla u)))+\epsilon \big)^{\alpha r}} \ dx \\
        &= \int_{B \cap \{A < 1\}} A(H(\nabla u)) |D^2u|^{2} \psi^{2} H_{\delta} \frac{A(H(\nabla u))}{\big(G_{\frac1N}(A(H(\nabla u)))+\epsilon \big)^{\alpha r}} \ dx \\
        &\quad + \int_{B \cap \{A \ge 1\}} A(H(\nabla u)) |D^2u|^{2} \psi^{2} H_{\delta} \frac{A(H(\nabla u))}{\big(G_{\frac1N}(A(H(\nabla u)))+\epsilon \big)^{\alpha r}} \ dx \\
        &\le \int_{B \cap \{A < 1\}} A(H(\nabla u)) |D^2u|^{2} \psi^{2} H_{\delta} \frac{1}{A(H(\nabla u))^{ \frac{r}{M_A}}} \ dx\\
        &\quad + \int_{B \cap \{A \ge 1\}} A(H(\nabla u)) |D^2u|^{2} \psi^{2} H_{\delta} \frac{A(H(\nabla u))H(\nabla u)^\sigma}{H(\nabla u)^\sigma} \ dx \\
        &\le  \frac{1}{A(1)^{\frac{r}{M_A}}} \int_{B \cap \{A < 1\}} \frac{ A(H(\nabla u)) |D^2u|^{2} \psi^{2} H_{\delta}}{\min \{H(\nabla u)^{r}, H(\nabla u)^{\frac{r m_A}{M_A}} \}} \ dx \\
        &\quad +\sup_{B}{|A(H(\nabla u))H(\nabla u)^\sigma|} \int_{B \cap \{A \ge 1\}} \frac{A(H(\nabla u)) |D^2u|^{2} \psi^{2} H_{\delta}} {H(\nabla u)^\sigma} \ dx \\ 
        &\le \hat{C} \cdot C(m_A,M_A,r,R,x_0,\gamma,\sigma,\|\nabla u\|_{L^{\infty}_{loc}(\Omega)}),\\
    \end{split} 
    \end{equation}
    where $\sigma <1$ is given in \eqref{sigma}, $C(m_A,M_A,r,R,x_0,\gamma,\sigma,\|\nabla u\|_{L^{\infty}_{loc}(\Omega)})$ is a positive constant and $\hat{C}$ is given by Theorem \ref{stime derivate seconde}. By \eqref{daino} and \eqref{daino1} we get 
    \begin{equation}\label{eq:I1peso}
        I_1 \le \theta C I_F+\frac{1}{4\theta}C(x_0,R,K_1,\alpha,r,\sigma,\gamma,m_A,M_A,\hat C,\|\nabla u\|_{L^{\infty}_{loc}(\Omega)}),
    \end{equation}
    with $C(\hat C, K_1,m_A,M_A,r,R,x_0,\alpha,\gamma, \sigma,\|\nabla u\|_{L^{\infty}_{loc}(\Omega)})$ positive constant.

Choosing $\theta$ sufficiently small such that $1-C\theta > 0,$ letting $\delta \rightarrow 0$ in \eqref{eq:conto}, by \eqref{eq:I3peso}, \eqref{eq:I2peso}, \eqref{eq:I1peso} and using Fatou's Lemma we get
    \begin{eqnarray}
        (1-C\theta) \int_{\Omega} \frac{1}{(G_{\frac1N}(A(H(\nabla u)))+\epsilon )^{\alpha r} |x-y|^{\gamma}} \psi^{2} \ dx  &\le& \Tilde{C},
    \end{eqnarray}
    where $\Tilde{C}=\tilde{C}(\hat C,K_1,m_{A},M_{A},r,R,x_0,\gamma,\sigma,\theta,\|\nabla u\|_{L^{\infty}_{loc}(\Omega)})$ is a positive constant.
    Using \eqref{eq:uniforme} and letting $\epsilon \rightarrow 0$  by Fatou Lemma we obtain
    \begin{equation}
        \int_{B_{R}(x_0)} \frac{1}{A(H(\nabla u))^{\alpha r} |x-y|^{\gamma}}  \ dx 
        \le C(\hat C,K_1,m_{A},M_{A},r,R,x_0,\gamma,\sigma,\theta,\|\nabla u\|_{L^{\infty}_{loc}(\Omega)}).
    \end{equation}
\end{proof}
 Before proving Theorem \ref{W2}, let us remark that:
 \begin{rem}\label{cianchirem}
Supposing $\displaystyle \inf_t A(t)>0$, the following proof of \eqref{eq:1} can easily simplified in \eqref{eq:primocaso2} in order to get $u\in W^{2,2}_{loc}(\Omega)$, without hypothesis on $M_A$.
\end{rem}

\begin{proof}[Proof of Theorem \ref{W2}]
    First we prove \eqref{eq:1}. Let 
    $x_0 \in \Omega$ and $R>0$ such that $B:=B_{R}(x_0) \subset \subset  \Omega.$ 
        Fix $\epsilon > 0$ and consider $G_{\epsilon}$ defined in \eqref{eq:Geps}.  We prove that
    $$ W_{\varepsilon,i} := G_{\varepsilon}(u_i) \in W^{1,2}_{loc}(\Omega),$$
uniformly in $\varepsilon$.  Let us remark that, for fixed $\varepsilon>0$,  since $u \in C^{2}(\Omega \setminus Z_u),$ and $G_{\varepsilon}(u_i)=0$ whenever $|u_i|\leq\varepsilon$ then $W_{\varepsilon,i}\in W^{1,\infty}_{loc}(\Omega)$. Moreover $\partial_{x_j}(W_{\varepsilon,i})=G'_{\varepsilon}(u_i) u_{ij}\chi_{\{|u_i|>\varepsilon\}}$ and
    \begin{equation}\label{2}
        \begin{split}
            \int_{B} \left|\nabla W_{\varepsilon,i} \right|^{2} \ dx \le 4 \int_{B \setminus Z_u} |D^2u|^{2} \ dx.
        \end{split}
    \end{equation}
Taking $\beta < 1$ we get
    \begin{equation}\label{eq:primocaso2}
    \begin{split}
         \int_{B \setminus Z_u} |D^2u|^{2} \ dx &= \int_{B \setminus Z_u} |D^2u|^{2} \frac{\min\{H(\nabla u)^{m_A}, H(\nabla u)^{M_A}\}}{|\nabla u|^{\beta}} \frac{|\nabla u|^{\beta}}{\min\{H(\nabla u)^{m_A}, H(\nabla u)^{M_A}\}} \ dx. \\
    \end{split}
    \end{equation}
    Using Theorem \ref{stime derivate seconde}, we can choose $\beta \approx 1^{-}.$ Since we are assuming $M_A < 1,$ then \eqref{eq:hp cianchi} and \eqref{H equiv euclidea} we get
    \begin{equation}\label{eq:w22}
    \begin{split}
        \int_{B \setminus Z_u}  |D^2u|^{2} \ dx &\le C(c_1,\|\nabla u\|_{L^{\infty}_{loc}(\Omega)}) \int_{B \setminus Z_u} \frac{\min\{H(\nabla u)^{m_A}, H(\nabla u)^{M_A} \}}{|\nabla u|^{\beta}} |D^2u|^{2} \ dx \\
        &\le C(c_1,\|\nabla u\|_{L^{\infty}_{loc}(\Omega)}) \int_{B \setminus Z_u} \frac{A(H(\nabla u)) |D^2u|^{2}}{|\nabla u|^{\beta}}  \ dx \le \hat{C} \cdot C(\|\nabla u\|_{L^{\infty}_{loc}(\Omega)}),
    \end{split}
    \end{equation}
    where $C(c_1,\|\nabla u\|_{L^{\infty}_{loc}(\Omega)})$ is a positive constant and  $\Hat{C}$ is given by Theorem $\ref{stime derivate seconde}$. Following the proof of Proposition \ref{eq:campovettoriale1} from \eqref{convdeb} we get the thesis.

    Now we suppose $M_A\geq 1$. 
    In this case,
      \begin{equation}
        \begin{split}
            \int_{B} \left|\nabla W_{\varepsilon,i} \right|^{q} \ dx \le 2^q \int_{B \setminus Z_u} |D^2u|^{q} \ dx.
        \end{split}
    \end{equation}    If we set $B_1:=B\cap \{H(\nabla u)<1\}$, using Holder inequality and \eqref{eq:hp cianchi} we can deduce 
    \begin{equation}
        \begin{split}
            \int_{B_1} |D^2u|^{q} \ dx &= \int_{B_1} \left( |D^2u|^{2} \right)^{\frac{q}{2}} \frac{A(H(\nabla u))^{\frac{q}{2}}}{|\nabla u|^{\beta \frac{q}{2}}} \frac{|\nabla u|^{\beta \frac{q}{2}}}{A(H(\nabla u))^{\frac{q}{2}}} \ dx \\
            &\le \left( \int_{B_1} \frac{A(H(\nabla u)) |D^2u|^{2}}{|\nabla u|^{\beta}} \ dx \right )^{\frac{q}{2}}   \left( \int_{B_1} \left( \frac{|\nabla u|^{\beta}}{A(H(\nabla u))} \right)^{\frac{q}{2-q}} \ dx \right)^{\frac{2-q}{2}} \\
            &\le C(c_1,\beta) \left( \int_{B} \frac{A(H(\nabla u)) |D^2u|^{2}}{|\nabla u|^{\beta}} \ dx \right )^{\frac{q}{2}} \left( \int_{B_1} \left( \frac{A(H(\nabla u))^{\frac{\beta}{M_A}}}{A(H(\nabla u))} \right)^{\frac{q}{2-q}} \ dx \right)^{\frac{2-q}{2}} \\
             &\le C(c_1,\beta) \left( \int_{B} \frac{A(H(\nabla u)) |D^2u|^{2}}{|\nabla u|^{\beta}} \ dx \right )^{\frac{q}{2}} \left( \int_{B_1} \frac{1}{A(H(\nabla u))^{\frac{q(M_A-\beta)}{M_A(2-q)}}}\ dx \right)^{\frac{2-q}{2}}. 
        \end{split}
    \end{equation}
    Applying Theorem \ref{stime derivate seconde} and Theorem \ref{inverso del peso} we have
    \begin{equation}\label{eq:B1}
        \int_{B_1} |D^2u|^{q} \ dx <+\infty.
    \end{equation}    On the other hand, since $u\in C^2(\Omega\setminus Z_u)$ we note that on the set  $B_2:=B\cap \{H(\nabla u)\geq1\}$, we have
\begin{equation}\label{eq:B2}
    \int_{B_2} |D^2u|^{q} \ dx <+\infty.
\end{equation}

By \eqref{eq:B1} and \eqref{eq:B2} we can conclude that $u\in W_{loc}^{2,q}(\Omega)$.
\end{proof}

\begin{center}{\bf Acknowledgements}\end{center}  
L. Muglia, and D. Vuono are members of INdAM. L. Muglia and D. Vuono are partially supported by PRIN project P2022YFAJH\_003 (Italy): Linear and nonlinear PDEs; new directions and applications.
\					
\begin{center}
 {\sc Data availability statement}\
All data generated or analyzed during this study are included in this published article.
\end{center}

					\
					
\begin{center}
{\sc Conflict of interest statement}
\
The authors declare that they have no competing interest.
\end{center}

\end{document}